\documentclass[10pt, dvipdfmx]{amsart}
\usepackage{subcaption}
\usepackage{amssymb}
\usepackage{graphicx}
\usepackage{here}
\usepackage{graphics}
\usepackage{amsmath}
\usepackage{amscd}
\usepackage{latexsym}
\usepackage[all]{xy}
\usepackage[percent]{overpic}
\begin{document}
\textwidth 5.5in
\textheight 8.3in
\evensidemargin .75in
\oddsidemargin.75in
\newtheorem{quest}{Question}[section]
\newtheorem{thm}[quest]{Theorem}
\newtheorem{lem}[quest]{Lemma}
\newtheorem{dfn}[quest]{Definition}
\newtheorem{conj}[quest]{Conjecture}
\newtheorem{cor}[quest]{Corollary}
\newtheorem{prop}[quest]{Proposition}
\newtheorem{prob}[quest]{Problem}
\newtheorem{claim}[quest]{Claim}
\newtheorem{exm}[quest]{Example}
\newtheorem{cond}{Condition}
\newtheorem{rmk}[quest]{Remark}
\newtheorem{que}[quest]{Question}
\newcommand{\p}[3]{\Phi_{p,#1}^{#2}(#3)}
\def\tu{\widetilde{\Upsilon}}
\def\Z{\mathbb Z}
\def\N{\mathbb N}
\def\C{\mathcal{C}}
\def\D{\mathcal{D}}
\def\R{\mathbb R}
\def\g{\overline{g}}
\def\odots{\reflectbox{\text{$\ddots$}}}
\newcommand{\tg}{\overline{g}}
\def\ee{\epsilon_1'}
\def\ef{\epsilon_2'}
\title{CONSTRUCTIONS OF HOMOTOPY 4-SPHERES BY POCHETTE SURGERY}
\author{TATSUMASA SUZUKI}
\thanks{The author is supported by JST SPRING, Grant Number JPMJSP2106}

\address{Department of Mathematics, Tokyo Institute of Technology, 2-12-1 Ookayama, Meguro-ku, Tokyo 152-8551, Japan}
\email{suzuki.t.do@m.titech.ac.jp}

\date{\today}


\maketitle
\begin{abstract}
The pochette surgery, which was discovered by Iwase and Matsumoto, is a generalization of the Gluck surgery.
In this paper we construct infinitely many embeddings of a pochette into the $4$-sphere and prove that homotopy $4$-spheres obtained from surgeries along these embedded pochettes are all diffeomorphic to the $4$-sphere.
\end{abstract}

\section{Introduction}
One of famous conjectures in $4$-manifold topology is the $4$-dimensional smooth Poincar\'e conjecture, which states that every homotopy $4$-sphere is diffeomorphic to the $4$-sphere.
A Gluck surgery on a $4$-manifold $X$ is an operation of removing the interior of a tubular neighborhood of a $2$-sphere in $X$ with trivial normal Euler number from $X$ and gluing $D^2\times S^2$ by a non-trivial diffeomorphism of the boundary $S^1\times S^2$. 
All Gluck surgeries on the $4$-sphere create homotopy $4$-spheres. 
Whether these manifolds are diffeomorphic to the $4$-sphere is a well-known unsolved problem.
The homotopy $4$-sphere obtained by the Gluck surgery along a spun $2$-knot or a $0$-slice $2$-knot is diffeomorphic to the $4$-sphere \cite{G, Me}.

In 2004, Iwase and Matsumoto \cite{IM} introduced a generalization of Gluck surgery called pochette surgery. 
A pochette is the boundary sum $P=S^1\times D^3\natural D^2\times S^2$ of $S^1\times D^3$ and $D^2\times S^2$. 
A pochette surgery on a $4$-manifold $X$ is an operation of removing the interior of $P$ embedded in $X$ from $X$ and gluing $P$ to $X-\mathrm{int}\, P$ by a diffeomorphism of $\partial P$.
The diffeomorphism type of the manifold $X'$ obtained by a pochette surgery along $P$ embedded in $X$ is determined by the embedding $e:P\rightarrow X$, an element $p/q$ of $\mathbb{Q}\cup\{\infty\}$ called the slope, and an element $\varepsilon$ of $\{0,1\}$ called the mod $2$ framing because the isotopy class of the gluing diffeomorphism of $\partial P$ is characterized by $p/q$ and $\varepsilon$. We denote $X'$ by $X(e,p/q,\varepsilon)$ and call it the pochette surgery on $X$ for $e,p/q,\varepsilon$. 
The manifold $X(e,1/0,1)$ is nothing but the Gluck surgery on $X$ for the embedded $2$-sphere $e(\{0\}\times S^2)$.

Let $DP=P\cup (-P)$ be the double of $P$ and $i_P:P\rightarrow DP$ the inclusion map. 
Kashiwagi \cite{K} found an algorithm for drawing handle diagrams of pochette surgeries of $DP$ and showed that $DP(i_P,1/q,\varepsilon)$ is diffeomorphic to the Pao manifold $L(q;0,1;\varepsilon)$(see \cite{P}). 
Murase \cite{Mu} constructed handle diagrams of all pochette surgeries of $DP$ and proved that $DP(i_P,p/q,\varepsilon)$ is diffeomorphic to $L(q;0,1;\varepsilon)$. 

We can consider $P$ as $h^0\cup h^1\cup h^2$, where $h^i$ is an $i$-handle for $i=0,1,2$.
Okawa \cite{O} proved that if the pochette surgery $S^4(e,p/q,\varepsilon)$ is a homology $4$-sphere and the core of $e(h^1)$ is `trivial' in $S^4-\mathrm{int}\, e(h^0\cup h^2)$, then $p$ must be $1$. 
He also showed that $S^4(e,1/q,\varepsilon)$ is diffeomorphic to the $4$-sphere if $e(\{0\}\times S^2)$ is a ribbon $2$-knot.

In this paper we construct infinitely many embeddings of $P$ into the $4$-sphere and prove that homotopy $4$-spheres obtained from surgeries along these embedded pochettes are all diffeomorphic to the $4$-sphere. 

The diagram depicted in Figure \ref{Theorem 1.1-1} is a handle diagram for the $4$-sphere, where $k$ is an integer greater than one, $n=(n_1,\ldots ,n_{k^2-1})$ is a $(k^2-1)$-tuple of integers, and the sign of $\pm 1$ can be taken arbitrarily. 
Let $e_{k,n}:P\rightarrow S^4$ be the inclusion map from the pochette $P$ which consists of the $0$-handle, the $1$-handle presented by the leftmost dotted circle, and the $2$-handle presented by the rightmost $0$-framed unknot in Figure \ref{Theorem 1.1-1}.
\begin{thm}
\label{pochette surgeries of S4 part1}
The pochette surgery $S^4(e_{k,n},1/q,\varepsilon)$ on the $4$-sphere $S^4$ for $e_{k,n}$, $1/q$, $\varepsilon$ is diffeomorphic to $S^4$ for every $k$, $n$, $q$ and $\varepsilon$ (See also Figure \ref{Theorem 1.1-2}). 
\end{thm}

\begin{figure}[htbp]
\begin{overpic}[scale=0.5]
{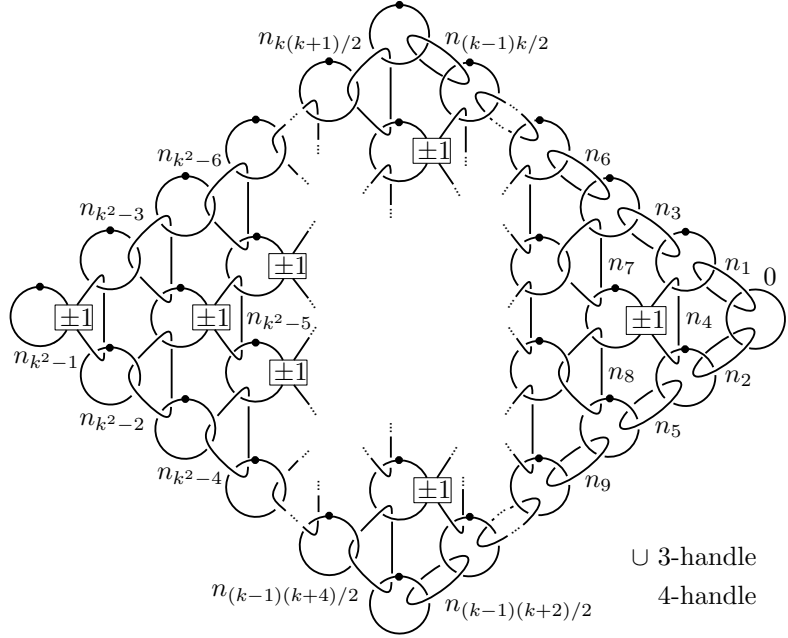}
\put(97,45){$0$}
\put(92,47){$n_1$}
\put(92,32.5){$n_2$}
\put(83,54){$n_3$}
\put(87,40){$n_4$}
\put(83,26){$n_5$}
\put(74,61){$n_6$}
\put(77,47){$n_7$}
\put(77,33){$n_8$}
\put(74,19){$n_9$}

\put(56,76){$n_{(k-1)k/2}$}
\put(56,3){$n_{(k-1)(k+2)/2}$}
\put(32,76){$n_{k(k+1)/2}$}
\put(26,5){$n_{(k-1)(k+4)/2}$}

\put(19,61){$n_{k^2-6}$}
\put(30.5,40){$n_{k^2-5}$}
\put(19,20){$n_{k^2-4}$}
\put(9,54){$n_{k^2-3}$}
\put(9,27){$n_{k^2-2}$}
\put(0.5,35){$n_{k^2-1}$}

\put(80,39.5){$\pm1$}
\put(52.5,61.3){$\pm1$}
\put(52.5,17.6){$\pm1$}
\put(34,46.5){$\pm1$}
\put(34,32.8){$\pm1$}
\put(24,39.8){$\pm1$}
\put(6.5,39.8){$\pm1$}

\put(80,9){$\cup$ 3-handle}
\put(83.4,4){4-handle}
\end{overpic}
   \caption{A handle diagram of the 4-sphere.}
   \label{Theorem 1.1-1}
\end{figure}

The diagram depicted in Figure \ref{Theorem 1.2-1} is also a handle diagram for the $4$-sphere, where $s, t$ are positive integers, 
$m=(m_1, \ldots, m_s)$ is a $s$-tuple of integers, and $n=(n_1, \ldots, n_{st+1})$ is a $(st+1)$-tuple of integers such that $\sum_{i=1}^{s}m_i=0$.  
Let $e_{m,n}:P\rightarrow S^4$ be the inclusion map from $P$ which consists of the $0$-handle, the $1$-handle presented by the leftmost dotted circle, and the $2$-handle presented by the rightmost $0$-framed unknot in Figure \ref{Theorem 1.2-1}. 

\begin{thm}
\label{pochette surgeries of S4 part2}
The pochette surgery $S^4(e_{m,n},1/q,\varepsilon)$ on the $4$-sphere $S^4$ for $e_{m,n}$, $1/q$, $\varepsilon$ is diffeomorphic to $S^4$ for every $m$, $n$, $q$ and $\varepsilon$ (See also Figure \ref{Theorem 1.2-2}). 
\end{thm}

\begin{figure}[htbp]
\begin{overpic}[scale=0.5]
{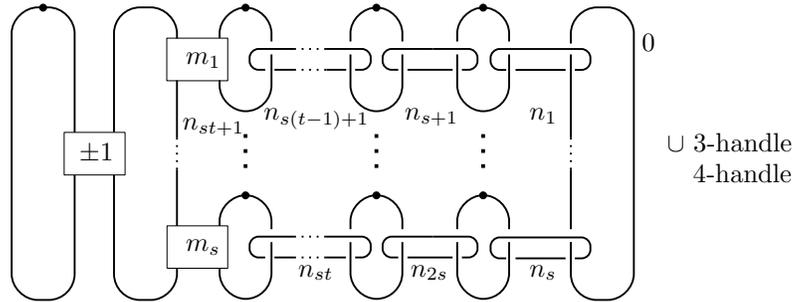}
\put(101,40){$0$}

\put(83,29){$n_1$}
\put(83,4){$n_s$}
\put(63,29){$n_{s+1}$}
\put(64,4){$n_{2s}$}

\put(40.5,29){$n_{s(t-1)+1}$}
\put(46,4){$n_{st}$}
\put(27.5,27.5){$n_{st+1}$}

\put(28,38){$m_1$}
\put(28,8){$m_s$}

\put(11,22.5){$\pm1$}

\put(105,24){$\cup$ 3-handle}
\put(109.2,19){4-handle}
\end{overpic}
   \caption{Another handle diagram of the 4-sphere.}
   \label{Theorem 1.2-1}
\end{figure}

\newpage
In Section 2 we review a precise definition and known properties of pochette surgery. 
In Section 3 we give proofs of the main results. 
We assume that all manifolds are smooth, compact and oriented and all maps are smooth.
\subsection*{Acknowledgements}
I am deeply grateful to my adviser, Hisaaki Endo for giving him courteous instructions in mathematics since I was a master's student.
I would also like to express my sincere gratitude to Motoo Tange of the University of Tsukuba for contributing to his knowledge of pochette surgery and giving the definition of the mod 2 framings and methods of handle calculus. 
Finally I want to thank Koji Yamazaki for suggesting the relationship between pochette surgery on the $4$-sphere and homotopy equivalence.
Furthermore I want to thank them for teaching us how to write a math paper. 

\begin{figure}[htbp]
\begin{overpic}[scale=0.5]
{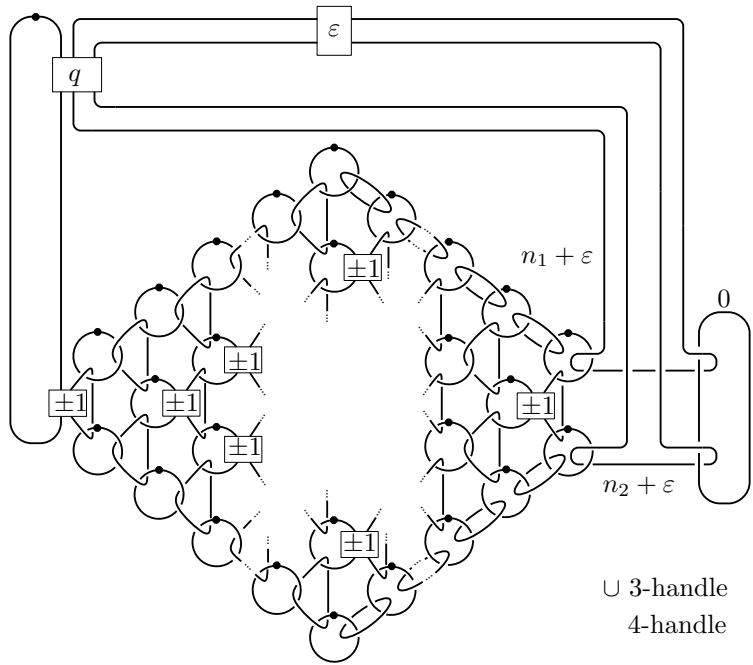}
\put(95.5,48){$0$}
\put(69,54){$n_1+\varepsilon$}
\put(80,23){$n_2+\varepsilon$}

\put(69,33.5){$\pm1$}
\put(45.5,52){$\pm1$}
\put(45,15){$\pm1$}
\put(29.5,39.5){$\pm1$}
\put(29.5,28){$\pm1$}
\put(21,34){$\pm1$}
\put(6,34){$\pm1$}

\put(8,78.5){$q$}
\put(43,84.5){$\varepsilon$}

\put(80,9){$\cup$ 3-handle}
\put(83.4,4){4-handle}
\end{overpic}
   \caption{The handle diagram of $S^4(e_{k, n}, 1/q, \varepsilon)$.}
   \label{Theorem 1.1-2}
\end{figure}

\begin{figure}[htbp]
\begin{overpic}[scale=0.5]
{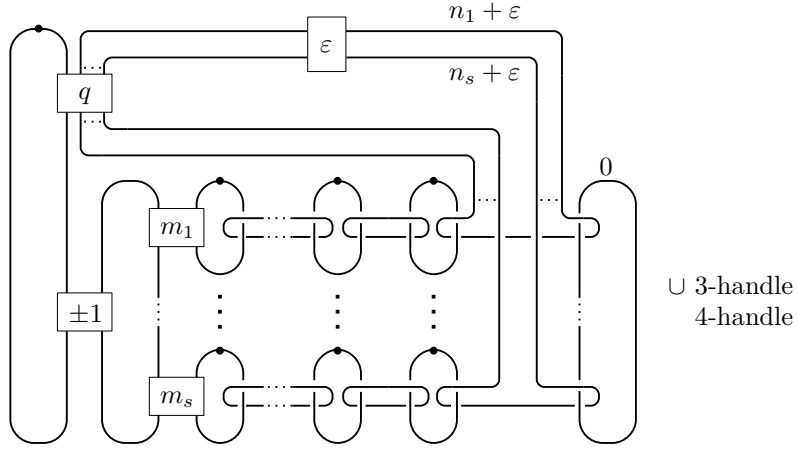}
\put(94,43){$0$}

\put(70,68){$n_1+\varepsilon$}
\put(70,58){$n_s+\varepsilon$}

\put(11,55){$q$}
\put(49.5,62.5){$\varepsilon$}

\put(24,33.5){$m_1$}
\put(24,6.5){$m_s$}

\put(9.5,19.5){$\pm1$}

\put(105,24){$\cup$ 3-handle}
\put(109.2,19){4-handle}
\end{overpic}
   \caption{The handle diagram of $S^4(e_{m, n}, 1/q, \varepsilon)$.}
   \label{Theorem 1.2-2}
\end{figure}


\section{Preliminaries}
\label{pre}
Let $X$ be a $4$-manifold and $E(A)$ the exterior $X-\mathrm{int}\, A$ of a subset $A$ of $X$.
Let $Q_e$ be the image $e(Q)$ of a subset $Q$ of $P$, $e:P\rightarrow X$ an embedding and $g:\partial{P}\rightarrow\partial E(P_e)$ a diffeomorphism. 
We call the curves $l:=S^1\times \{*\}$ and $m:=\partial D^2\times \{*\}$ on $\partial P$ a {\it longitude} and a {\it meridian} of $P$, respectively. 
First, we define the pochette surgery. 

\begin{dfn}[Iwase-Matsumoto \cite{IM}]
\label{pochete surgery}
\rm A {\it pochette surgery} on $X$ is an operation of removing $\mathrm{int}\, P_e$ and gluing in $P$ by $g:\partial{P}\rightarrow\partial E(P_e)$. 
The 4-manifold $E(P_e)\cup_{g}P$ obtained by the pochette surgery on $X$ using $e$ and $g$ is denoted by $X(e, g)$. 
The manifold $X(e, g)$ is also called the pochette surgery on $X$ for $e$ and $g$. 
\end{dfn}

In pochette surgery on a 4-manifold, after attaching $D^2\times S^2$ to $P$ along $g(m)$, the method of attaching $S^1\times D^3$ is unique.
Therefore, when gluing $P$, it is sufficient to consider an identification between neighborhoods of $m$ and $g(m)$ via $g$.

Fix an identification between $\partial P$ and $S^1\times \partial D^3\# \partial D^2\times S^2=S^1\times S^2\#S^1\times S^2$.
The meridian $m$ of $P$ has the natural product framing.
By embedding $e$, we get identification $\iota:\partial E(P_e)\to S^1\times S^2\#S^1\times S^2$.
Then, $S^1\times S^2\# S^1\times S^2$ can be expressed as the 2-component unlink which consists of 2 0-framed knots.
Therefore, $g$ maps the natural framing on $m$ of $\partial P$ to a framing on $g(m)$.
This framing on $g(m)$ is represented by some integer determined by $\iota$.
The pochette can be regarded as $S^1\times D^3$ attaching a 2-handle with the cocore $m$.
Let $g_1, g_2:\partial P\to \partial E(P_e)$ be two gluing maps. 
If $g_1(m)$ and $g_2(m)$ are the same and a difference between the framing on $g_1(m)$ and that of $g_2(m)$ is even, the map $g_1^{-1}\circ g_2|_{N(m)}$ can be extended to the inside of the 2-handle.
Here, $N(A)$ is the open tubular neighborhood for a submanifold $A$ of $P$.
Therefore, when considering the diffeomorphism type of the pochette surgery, we should consider an integer modulo 2 as the framing on $g(m)$.
This framing on $g(m)$ is called a {\it mod 2 framing} and write it as $\varepsilon$.
The mod 2 framing of $g(m)$ for the gluing map $g:\partial P\to \partial E(P_e)$ was first introduced in \cite[First paragraph in p.162]{IM}.

By \cite[Lemma 4]{IM}, the diffeomorphism type of $X(e, g)$ is determined by the embedding $e:P\rightarrow X$, the isotopy class of a simple closed curve $g(m)$ and the mod 2 framing around $g(m)$. 
For orientation preserving self-diffeomorphisms $g,g'$ of $\partial P$, if $g_*([m])$ is equal to $g'_*([m])$, then $g(m)$ is isotopic to $g'(m)$ (see [IM, Lemma 5]).
Hence, the diffeomorphism type of $X(e, g)$ is determined by an embedding $e:P\rightarrow X$, a homology class $g_{*}([m])$ in $H_1(\partial E(P_e))\cong\mathbb{Z}[m]\oplus\mathbb{Z}[l]$ and the mod 2 framing around $g(m)$.

Let $p, q$ be coprime integers and $\varepsilon$ an element of $\{0, 1\}$. 
By \cite[the seventh paragraph in p.\ 163]{IM}, the homology class $p[m]+q[l]\in \mathbb{Z}[m]\oplus\mathbb{Z}[l]$ is determined by $p/q\in \mathbb{Q}\cup\{\infty\}$ up to sign of $p$. 
The next theorem immediately follows from observations here (see [IM, Theorem 2]).

\begin{thm}[Iwase-Matsumoto \cite{IM}]
\label{three conditions 3}
The diffeomorphism type of $X(e, g)$ is determined by the following data: \\
(1)an embedding $e:P\rightarrow X$, \\
(2)a slope p/q, \\
(3)a mod 2 framing around $g(m)$.
\end{thm}

Let $g_{p/q,\varepsilon}:\partial{P}\rightarrow \partial E(P_e)$ be a diffeomorphism which satisfies $g_{p/q,\varepsilon *}([m])=p[m]+q[l]$ and the mod 2 framing of $g_{p/q,\varepsilon}(m)$ is $\varepsilon$ in $\{0, 1\}$. 
We can define $X(e, p/q, \varepsilon)=X(e, g_{p/q, \varepsilon})$. 
From the construction, any pochette surgery for $e, 1/0, \varepsilon$ is nothing but the Gluck surgery along $S_e$, where $S$ is the subset $\{0\}\times S^2$ of $P$. 

Suppose $X$ is a homology $4$-sphere, and $i_{11}:H_1(\partial P) \rightarrow H_1(E(P_e))$ is the induced homomorphism of the composite map of the gluing map $g:\partial P\rightarrow \partial E(P_e)$ and the inclusion map $i_{\partial E(P_e)}:\partial E(P_e)\hookrightarrow E(P_e)$. 
Okawa calculated some homology groups of the pochette surgery $X(e, p/q, \varepsilon)$: 
\begin{thm}[Okawa {\cite[Theorem 1.1]{O}}]
\label{homologyS4}
If $i_{11}([l])=0$, then $H_1(X(e, p/q, \varepsilon))\cong\mathbb{Z}_{p}$ and $H_2(X(e, p/q, \varepsilon))=0.$
Moreover, if $|p|$ is equal to $1$, then $X(e, p/q, \varepsilon)$ is a homology 4-sphere. 
\end{thm}

Homology groups of pochette surgeries on any simply-connected closed $4$-manifold will be calculated in Section \ref{Homology of pochette surgery}. 

Recall $P$ can be interpreted as $h^0\cup h^1\cup h^2$. 
We call the core of $(h^1)_e$ a {\it cord}. 
A cord is {\it trivial} if it is boundary parallel. 
Let $S:=\{*\}\times S^2\subset P$. 
Okawa also showed that $S^4(e,1/q,\varepsilon)$ is diffeomorphic to the $4$-sphere if the cord $(h^1)_e$ is trivial in $E((h^0\cup h^2)_e)$ and $S_e$ is a ribbon $2$-knot. 

\begin{thm}[Okawa {\cite[Theorem 1.2]{O}}]
If $(h^1)_e$ is the trivial code and $S_e$ is a ribbon 2-knot, then the manifold $S^4(e, 1/q, \varepsilon)$ is diffeomorphic to the $4$-sphere.
\end{thm}
\begin{proof}
This is a variation of \cite[Exercise 6.2.11(b)]{GS}. 
\end{proof}


\section{Handle diagram for pochette surgery}
\label{Handle diagram for pochette surgery}

In this section we give a construction of handle diagrams for pochette surgeries under special conditions. 
Let $X$ be a $4$-manifold and $e:P\rightarrow X$ an embedding from a pochette $P$ into $X$. 
Let $p,q$ be coprime integers and $\varepsilon$ an element of $\{0,1\}$. 
Suppose that the diagram depicted in Figure \ref{Proposition 3.1-1} is a part of a handle diagram for $X$, 
where all the curves partially drawn in Figure \ref{Proposition 3.1-1} are framed knots, and any framed knot entwined with the dotted circle in Figure \ref{Proposition 3.1-1} has a $0$-framed meridian.
The pochette $P_e$ consists of the $0$-handle, the $1$-handle presented by the leftmost dotted circle, and the $2$-handle presented by the rightmost $0$-framed unknot in Figure \ref{Proposition 3.1-1}.

\begin{prop}
\label{recipe of pochette surgery}
A handle diagram for $X(e,p/q,\varepsilon)$ is depicted in Figure \ref{Proposition 3.1-2}(A) if $p/q$ is not equal to $0/1$, and that for $X(e,0/1,\varepsilon)$ is depicted in Figure \ref{Proposition 3.1-2}(B). 
\begin{figure}[htbp]
   \begin{overpic}[scale=0.5]
   {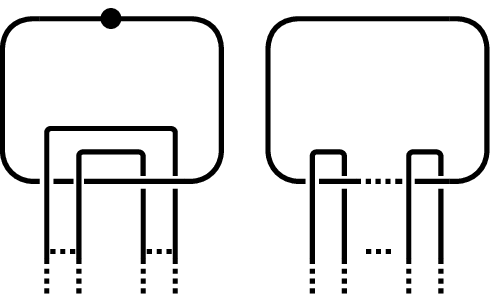}
   \put(73,44){$0$}
   \end{overpic}
   \caption{A part of the handle diagram of $X$.}
   \label{Proposition 3.1-1}
    \begin{tabular}{cc}
      \begin{minipage}[t]{0.45\hsize}
        \centering
        \begin{overpic}[scale=0.5]
        {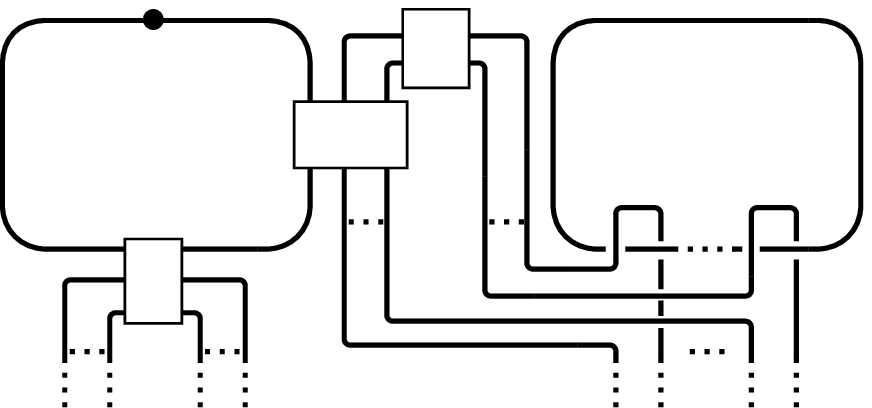}
        \put(80,38){$0$}

        \put(49,40){$\varepsilon$}
        \put(39,30){$q$}
        \put(16,14){$p$}
        \end{overpic}
        \subcaption{$p/q\neq 0/1$}
      \end{minipage} 
      \begin{minipage}[t]{0.45\hsize}
        \centering
        \begin{overpic}[scale=0.5]
        {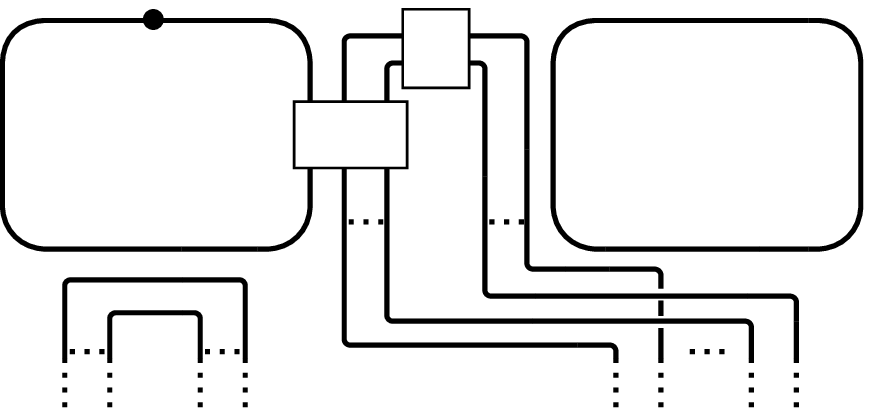}
        \put(80,38){$0$}

        \put(49,40){$\varepsilon$}
        \put(39,29){$1$}
        \end{overpic}
        \subcaption{$p/q=0/1$}
      \end{minipage}
    \end{tabular}
     \caption{The handle diagram of $X(e, p/q, \varepsilon)$.}
     \label{Proposition 3.1-2}
\end{figure}
\end{prop}

\begin{proof}
Here we will consider the case where only a framed knot is entwined with the 0-framed knot on the right side exactly once. The case where framed knots are entwined with the 0-framed knot on the right side can be proved in the same way. 
If $|p|$ and $|q|$ are coprime positive integers, then there exist a positive integer $n$, a nonnegative integer $a_0$ and positive integers $a_1,\ldots, a_n$ such that 
\[\frac{|p|}{|q|}=a_0+\cfrac{1}{a_1+\cfrac{1}{\ddots+\cfrac{1}{a_n}}}.\]

We define the diffeomorphism $E_0, E_1, E_2, E_3, E_4$ and $E_5:\partial P \rightarrow\partial P$ as the $1$-Rolfsen twist for the leftmost $\langle0\rangle$-framed knot, the handle slide in Figure \ref{E1}, \ref{E2}, \ref{E3} and \ref{E4}, the operation changing the direction of the meridian $m$, respectively. 

\begin{figure}[htbp]
\begin{minipage}[b]{0.45\linewidth}
\centering
\begin{overpic}[scale=0.5]
{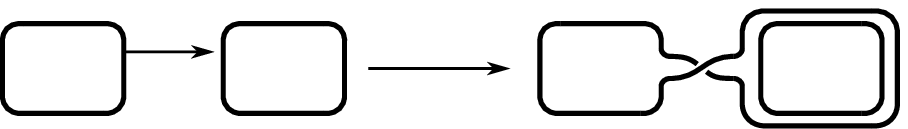}
\put(02,15){$\langle0\rangle$}
\put(27,15){$\langle0\rangle$}

\put(41.5,9){slide}

\put(62,15){$\langle0\rangle$}
\put(86,5){$\langle0\rangle$}
\end{overpic}
\caption{$E_1$.}
\label{E1}
\end{minipage}
\begin{minipage}[b]{0.45\linewidth}
\begin{overpic}[scale=0.5]
{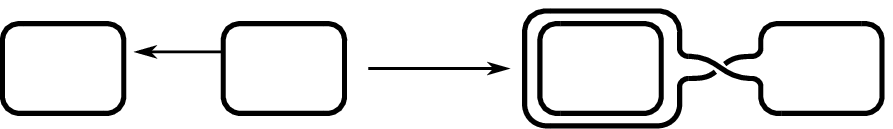}
\put(02,15){$\langle0\rangle$}
\put(27,15){$\langle0\rangle$}

\put(41.5,9){slide}

\put(63,5){$\langle0\rangle$}
\put(88,15){$\langle0\rangle$}
\end{overpic}
\caption{$E_2$.}
\label{E2}
\end{minipage}
\end{figure}
\begin{figure}[htbp]
\begin{minipage}[b]{0.45\linewidth}
\centering
\begin{overpic}[scale=0.5]
{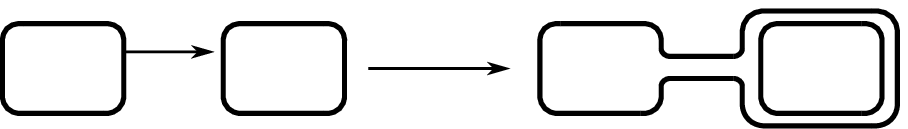}
\put(02,15){$\langle0\rangle$}
\put(27,15){$\langle0\rangle$}

\put(41.5,9){slide}

\put(62,15){$\langle0\rangle$}
\put(86,5){$\langle0\rangle$}
\end{overpic}
\caption{$E_3$.}
\label{E3}
\end{minipage}
\begin{minipage}[b]{0.45\linewidth}
\begin{overpic}[scale=0.5]
{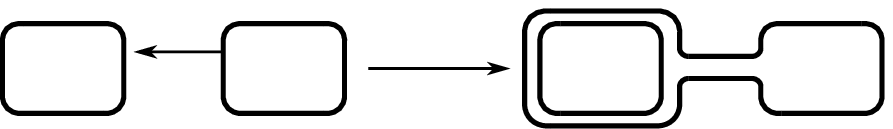}
\put(02,15){$\langle0\rangle$}
\put(27,15){$\langle0\rangle$}

\put(41.5,9){slide}

\put(63,5){$\langle0\rangle$}
\put(88,15){$\langle0\rangle$}
\end{overpic}
\caption{$E_4$.}
\label{E4}
\end{minipage}
\end{figure}

Then we have 
\[E_{i*}([m])=
\begin{cases}
   [m]&(i=0, 2, 4), \\
   [m]+[l]&(i=1), \\
   [m]-[l]&(i=3), \\
   -[m]&(i=5),
\end{cases}
E_{i*}([l])=
\begin{cases}
   [l]&(i=0, 1, 3, 5), \\
   [m]+[l]&(i=2), \\
   -[m]+[l]&(i=4).
\end{cases}
\]
We define $E_{p/q, \varepsilon}$ to be  
\[\begin{cases}
E_0^{\varepsilon}&(p=1, q=0), \\
E_4\circ E_1\circ E_0^{\varepsilon}&(p=0, q=1), \\
E_5\circ E_0^{\varepsilon}&(p=-1, q=0), \\
E_4\circ E_1\circ E_5\circ E_0^{\varepsilon}&(p=0, q=-1), \\
E_2^{a_0}\circ E_1^{a_1}\circ E_2^{a_2}\circ\cdots\circ E_1^{a_{n-1}}\circ E_2^{a_{n}-1}\circ E_1\circ E_0^{\varepsilon}&(p>0, q>0, n\in 2\mathbb{N}), \\
E_2^{a_0}\circ E_1^{a_1}\circ E_2^{a_2}\circ\cdots\circ E_2^{a_{n-1}}\circ E_1^{a_{n}}\circ E_0^{\varepsilon}&(p>0, q>0, n\in 2\mathbb{N}-1), \\
E_4^{a_0}\circ E_3^{a_1}\circ E_4^{a_2}\circ\cdots\circ E_3^{a_{n-1}}\circ E_4^{a_{n}-1}\circ E_3\circ E_0^{\varepsilon}&(p>0, q<0, n\in 2\mathbb{N}), \\
E_4^{a_0}\circ E_3^{a_1}\circ E_4^{a_2}\circ\cdots\circ E_4^{a_{n-1}}\circ E_3^{a_{n}}\circ E_0^{\varepsilon}&(p>0, q<0, n\in 2\mathbb{N}-1), \\
E_4^{a_0}\circ E_3^{a_1}\circ E_4^{a_2}\circ\cdots\circ E_3^{a_{n-1}}\circ E_4^{a_{n}-1}\circ E_3\circ E_0^{\varepsilon}&(p<0, q>0, n\in 2\mathbb{N}), \\
E_4^{a_0}\circ E_3^{a_1}\circ E_4^{a_2}\circ\cdots\circ E_4^{a_{n-1}}\circ E_3^{a_{n}}\circ E_5\circ E_0^{\varepsilon}&(p<0, q>0, n\in 2\mathbb{N}-1), \\
E_2^{a_0}\circ E_1^{a_1}\circ E_2^{a_2}\circ\cdots\circ E_1^{a_{n-1}}\circ E_2^{a_{n}-1}\circ E_1\circ E_5\circ E_0^{\varepsilon}&(p<0, q<0, n\in 2\mathbb{N}), \\
E_2^{a_0}\circ E_1^{a_1}\circ E_2^{a_2}\circ\cdots\circ E_2^{a_{n-1}}\circ E_1^{a_{n}}\circ E_5\circ E_0^{\varepsilon}&(p<0, q<0, n\in 2\mathbb{N}-1)
\end{cases}\]
and $g_{p/q, \varepsilon}=e\circ E_{p/q, \varepsilon}$. 
Then we have $E_{p/q, \varepsilon\ast}([m])=p[m]+q[l]$ for any $p/q\in\mathbb{Q}\cup\{\infty\}$ and $\varepsilon\in\{0,1\}$. 
By Theorem \ref{three conditions 3}, the pochette surgery $X(e, p/q, \varepsilon)$ is diffeomorphic to $X(e, g_{p/q, \varepsilon})$ for any $p/q\in\mathbb{Q}\cup\{\infty\}$ and $\varepsilon\in\{0,1\}$. 
A part of a handle diagram of $X$ is depicted in Figure \ref{Proposition 3.1proof1}. 
By several handle slides on the 0-framed meridians of the framed knots entwined with the dotted circle, we obtain a part of a handle diagram of $X(e, p/q, \varepsilon)$ depicted in Figure \ref{Proposition 3.1proof2}. 
Concretely, the homotopy class of $g_{p/q, \varepsilon}(m)$ is the natural lift defined in \cite{IM}: 
\[\begin{cases}
m'^pl'^q&(pq=0), \\
\prod_{k=1}^{|p|}l'^{q/|q|(\lfloor k|q|/|p|\rfloor-\lfloor(k-1)|q|/|p|\rfloor)}m'^{p/|p|}&(pq\neq0). 
\end{cases}\]
Here, $l',m'$ are the images on $\pi_1(\partial E(P_e))$ of based, oriented, longitude and meridian in $\partial P$ via $e$.

If $pq = 0$, then we reach the desired result.

If $pq\neq0$, we obtain Figure \ref{Proposition 3.1proof3} by creating a $2$-handle/$3$-handle pair in Figure \ref{Proposition 3.1proof2}.
By the handle slide in Figure \ref{Proposition 3.1proof3} and several handle slides on the 0-framed meridians of the framed knots entwined with the dotted circle, we obtain the handle diagram depicted in Figure \ref{Proposition 3.1proof4}.
By the handle slides between the leftmost 0-framed knot and the rightmost 0-framed knot in Figure \ref{Proposition 3.1proof4}:
\[\begin{cases}
E_0^{\varepsilon}\circ E_{p/q,\varepsilon}^{-1}&(p>0), \\
E_5\circ E_0^{\varepsilon}\circ E_{p/q,\varepsilon}^{-1}&(p<0),
\end{cases}\]
we obtain the handle diagram depicted in Figure \ref{Proposition 3.1proof5}.
By the handle slide on a 0-framed meridian in Figure \ref{Proposition 3.1proof5}, the other 0-framed meridian in Figure \ref{Proposition 3.1proof5} can be changed to 0-framed unlink. 
Changing the self-intersection of the framed knot in Figure \ref{Proposition 3.1proof5} by several handle slides on the 0-framed meridian and canceling the $2$-handle/$3$-handle pair, we obtain the handle diagram depicted in Figure \ref{Proposition 3.1proof6}. 
Therefore, we also obtain the conclusion in the case of $pq\neq0$. 
\end{proof}

\begin{figure}[htbp]
  \begin{minipage}[b]{0.45\linewidth}
    \centering
     \begin{overpic}[scale=0.5]
     {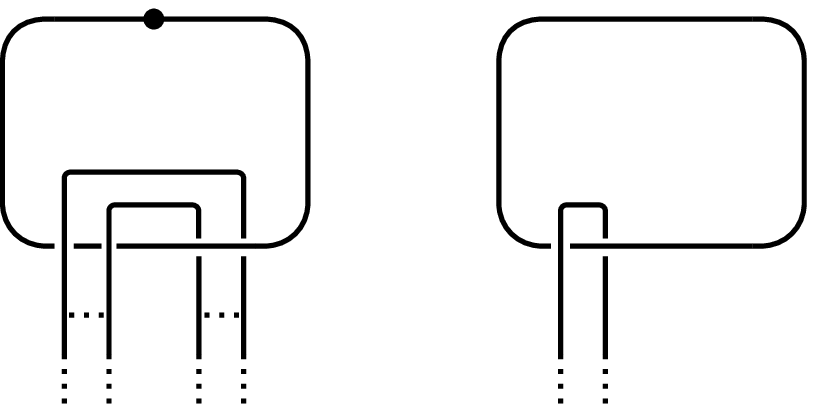}
     \put(78,40){$0$}
     \end{overpic}
    \caption{}
    \label{Proposition 3.1proof1}
  \end{minipage}
  \begin{minipage}[b]{0.45\linewidth}
    \centering
     \begin{overpic}[scale=0.5]
     {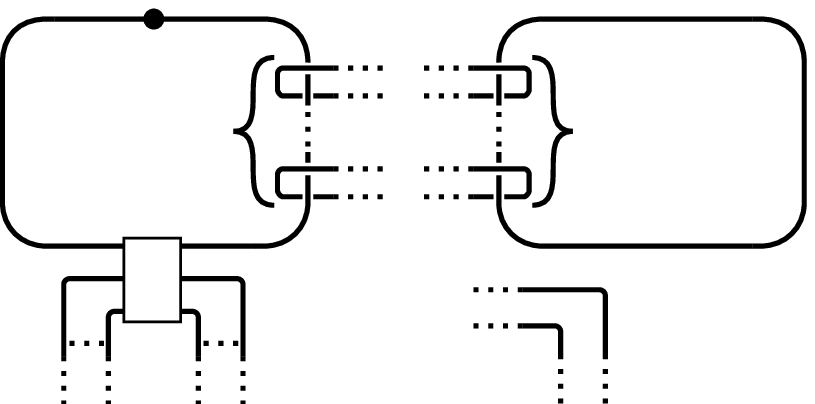}
     \put(17,14){$p$}
     \put(20,32){$|q|$}
     \put(71.5,32){$|p|$}

     \put(78,40){$0$}
     \end{overpic}
    \caption{}
    \label{Proposition 3.1proof2}
  \end{minipage}
\end{figure}
\begin{figure}[htbp]
  \begin{minipage}[b]{0.45\linewidth}
    \centering
     \begin{overpic}[scale=0.5]
     {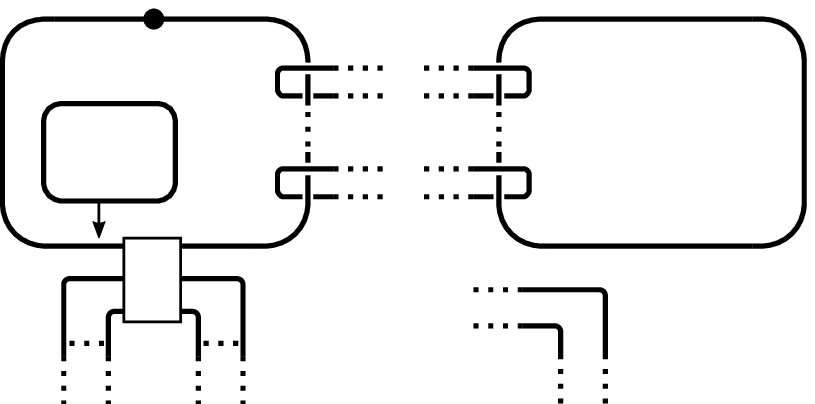}
     \put(11.5,39){$0$}
     \put(17,14){$p$}

     \put(78,40){$0$}
     \end{overpic}
    \caption{}
    \label{Proposition 3.1proof3}
  \end{minipage}
  \begin{minipage}[b]{0.45\linewidth}
    \centering
     \begin{overpic}[scale=0.5]
     {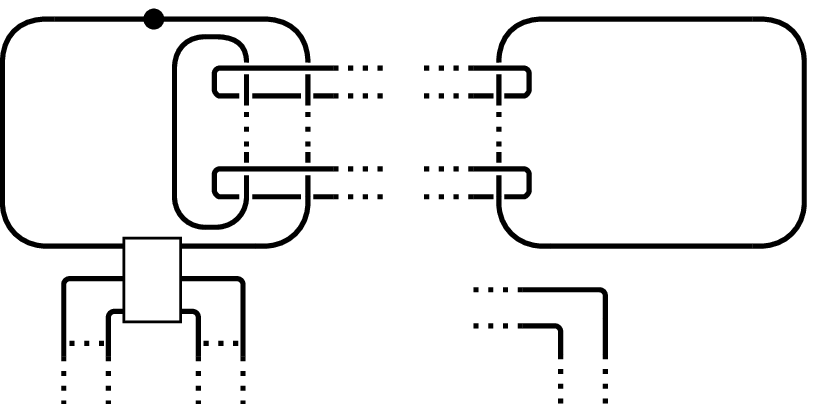}
     \put(17,14){$p$}
     \put(15,32){$0$}

     \put(78,40){$0$}
     \end{overpic}
    \caption{}
    \label{Proposition 3.1proof4}
  \end{minipage}
\end{figure}
\begin{figure}[htbp]
  \begin{minipage}[b]{0.45\linewidth}
    \centering
     \begin{overpic}[scale=0.5]
     {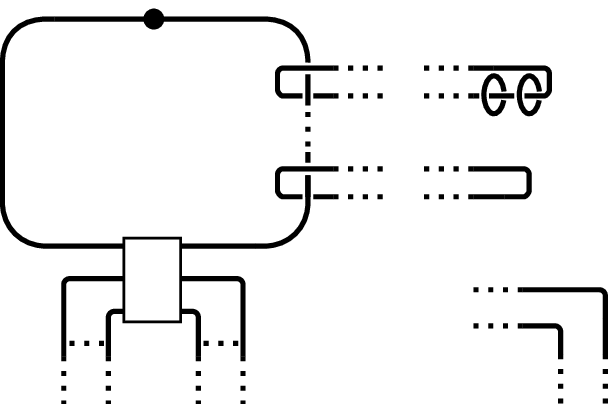}
     \put(78,58){$0$}
     \put(85,58){$0$}
     \put(22,18){$p$}
     \end{overpic}
    \caption{}
    \label{Proposition 3.1proof5}
  \end{minipage}
  \begin{minipage}[b]{0.45\linewidth}
    \centering
     \begin{overpic}[scale=0.5]
     {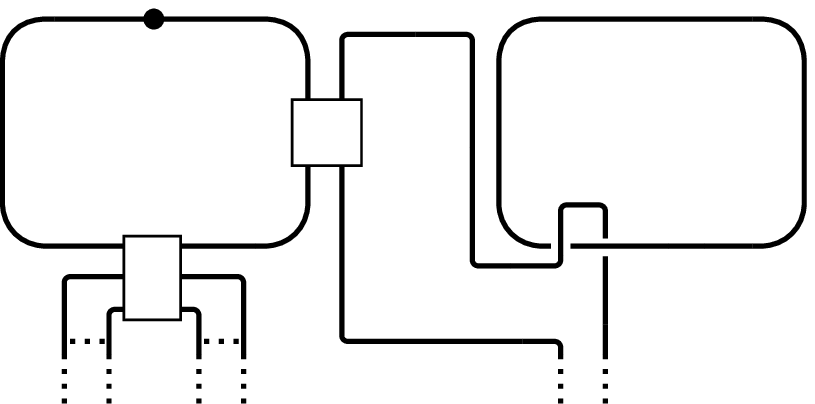}
     \put(17,14){$p$}
     \put(38.5,32.5){$q$}

     \put(78,40){$0$}
     \end{overpic}
    \caption{}
    \label{Proposition 3.1proof6}
  \end{minipage}
\end{figure}

\newpage
\begin{rmk}
\label{a generalization of prop}
\rm If $p=1$, Proposition \ref{recipe of pochette surgery} holds even without the 0-framed meridians, so any pochette surgery $X(e, 1/q, \varepsilon)$ is given by Figure \ref{Proposition 3.1-2}(A). 
\end{rmk}

\begin{rmk}
\rm Proposition \ref{recipe of pochette surgery} is a generalization of \cite[Theorem 1, 2]{K} and \cite[Theorem 1.1, 1.2]{Mu}. 
The curve $g_{p/q,\varepsilon}(m)$ in $\partial P$ is depicted in Figure \ref{remark1} in the case of $|p|=1, q=0$, in Figure \ref{remark2} in the case of $p=0, |q|=1$, in Figure \ref{remark3} in the case of $|q|>|p|>0, pq>0$, in Figure \ref{remark4} in the case of $|q|>|p|>0, pq<0$, in Figure \ref{remark5} in the case of $|p|>|q|>0, pq>0$, in Figure \ref{remark6} in the case of $|p|>|q|>0, pq<0$. 
We perform a Rolfsen twist just before or just after performing each handle slide $E_1,E_2,E_3,E_4$ by using either of two $\langle 0\rangle$-framed knots in $\partial P_e$.
Then we can obtain the curves depicted in Figure \ref{remark3}--\ref{remark6}.
The curve depicted in Figure \ref{remark3} was first discovered by Murase \cite{Mu}.

\begin{figure}[htbp]
  \begin{minipage}[b]{0.45\linewidth}
    \centering
     \begin{overpic}[scale=0.5]
     {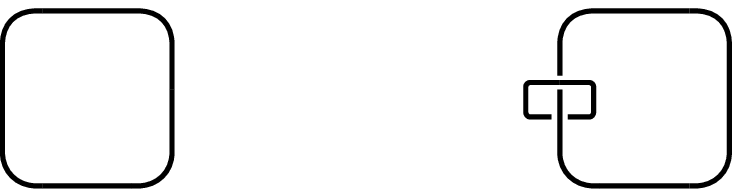}
     \put(6,15){$\langle0\rangle$}
     \put(82,15){$\langle0\rangle$}
     \end{overpic}
    \caption{}
    \label{remark1}
  \end{minipage}
  \begin{minipage}[b]{0.45\linewidth}
    \centering
     \begin{overpic}[scale=0.5]
     {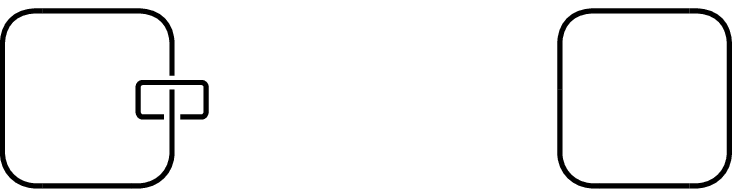}
     \put(6,15){$\langle0\rangle$}
     \put(82,15){$\langle0\rangle$}
     \end{overpic}
    \caption{}
    \label{remark2}
  \end{minipage}
\end{figure}
\begin{figure}[htbp]
  \begin{minipage}[b]{0.45\linewidth}
    \centering
     \begin{overpic}[scale=0.5]
     {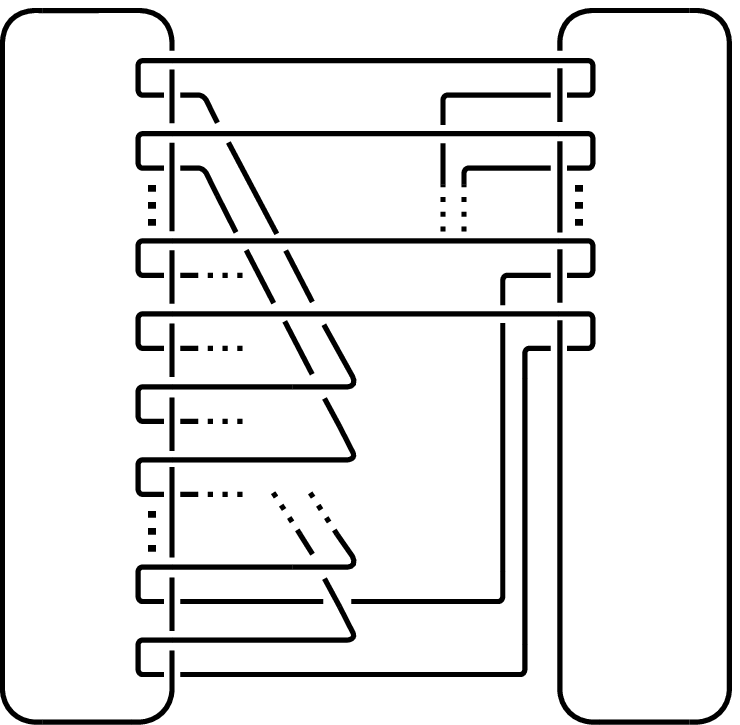}
     \put(6,88){$\langle0\rangle$}
     \put(82,88){$\langle0\rangle$}
     \end{overpic}
    \caption{}
    \label{remark3}
  \end{minipage}
  \begin{minipage}[b]{0.45\linewidth}
    \centering
     \begin{overpic}[scale=0.5]
     {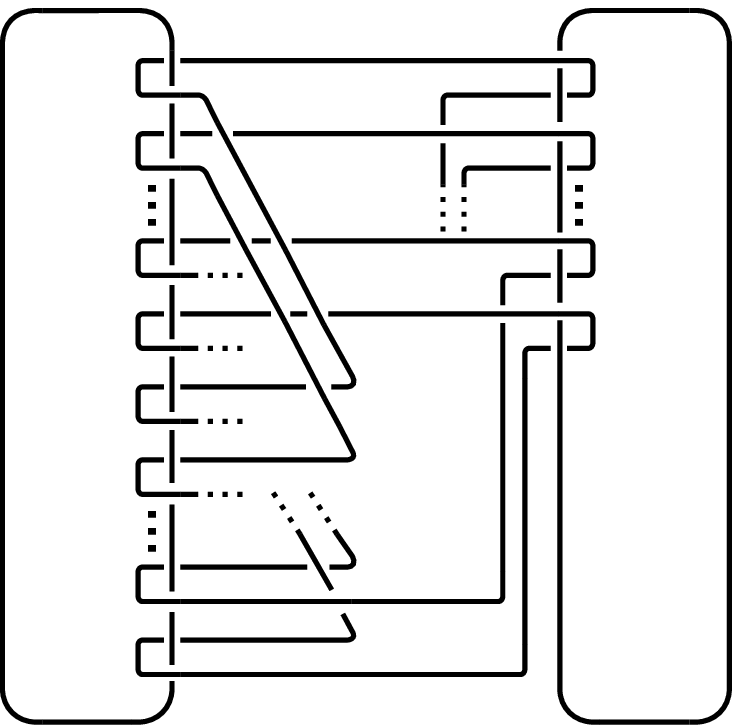}
     \put(6,88){$\langle0\rangle$}
     \put(82,88){$\langle0\rangle$}
     \end{overpic}
    \caption{}
    \label{remark4}
  \end{minipage}
\end{figure}
\begin{figure}[htbp]
  \begin{minipage}[b]{0.45\linewidth}
    \centering
     \begin{overpic}[scale=0.5]
     {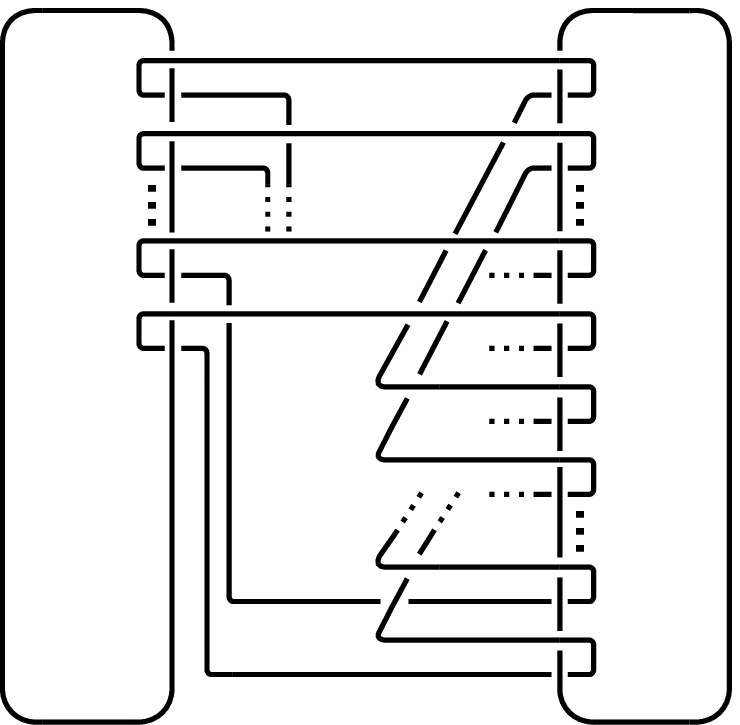}
     \put(6,88){$\langle0\rangle$}
     \put(82,88){$\langle0\rangle$}
     \end{overpic}
    \caption{}
    \label{remark5}
  \end{minipage}
  \begin{minipage}[b]{0.45\linewidth}
    \centering
     \begin{overpic}[scale=0.5]
     {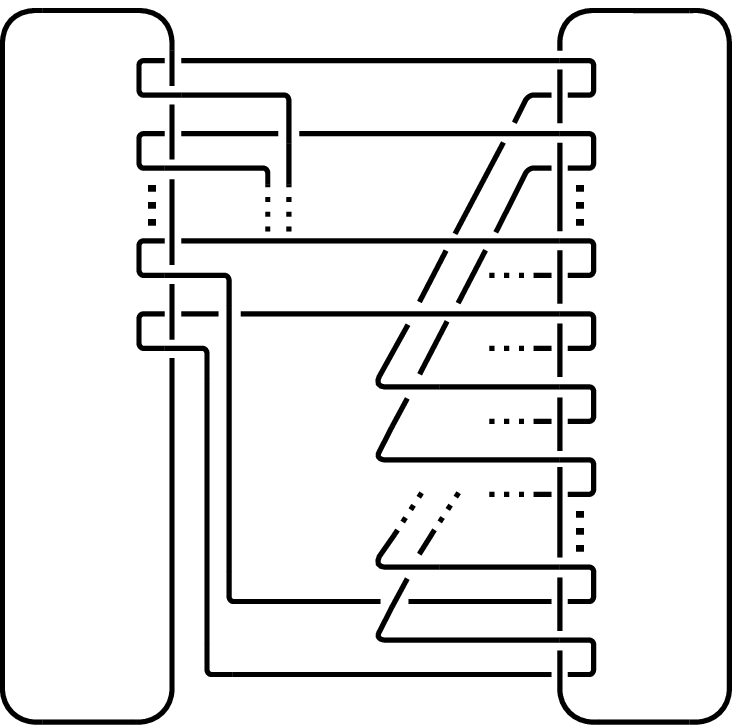}
     \put(6,88){$\langle0\rangle$}
     \put(82,88){$\langle0\rangle$}
     \end{overpic}
    \caption{}
    \label{remark6}
  \end{minipage}
\end{figure}
\end{rmk}

\newpage
\section{Homology of pochette surgery}
\label{Homology of pochette surgery}

Let $p,q$ be coprime integers and $\varepsilon$ an element of $\{0,1\}$.
Let $B:=\{*\}\times \partial D^3\subset P$. 
Let $X$ be a simply-connected closed $4$-manifold and $e:P\rightarrow X$ an embedding from a pochette $P$ into $X$. 
Here, we prove lemmas needed later. 

\begin{lem}
\label{homology EPe}
If the homomorphism $t_2:H_2(X)\rightarrow H_2(X,E(P_e))$ induced by the inclusion map $(X,\varnothing)\rightarrow (X,E(P_e))$ is a zero map, the homology groups of the exterior $E(P_e)$ of $P_e$ are calculated as follows:
\[H_n(E(P_e))=
\begin{cases}
   {\mathbb Z}[x_e]&(n=0), \\
   {\mathbb Z}[m_e]&(n=1), \\
   {\mathbb Z}[B_e]\oplus H_2(X)&(n=2), \\
   0&(otherwise). 
\end{cases}\]
Here $x$ is a point in $\partial P$.
\end{lem}

\begin{proof}
By the long exact sequence of the pair $(P, \partial P)$:
\[\cdots\overset{\partial_{n+1}}{\longrightarrow} H_n(\partial P) \overset{s_{n}}{\longrightarrow} H_n(P) \overset{t_{n}}{\longrightarrow} H_n(P, \partial P) \overset{\partial_{n}}{\longrightarrow}\cdots, \]
we have 
\[H_n(P, \partial P)=
\begin{cases}
   {\mathbb Z}[D^2\times\{*\}]&(n=2), \\
   {\mathbb Z}[\{*\}\times D^3]&(n=3), \\
   {\mathbb Z}[P]&(n=4), \\
   0&(otherwise). 
\end{cases}\]
By the Excision Theorem, we obtain 
\[H_n(X, E(P_e))\cong H_n(P, \partial P)\ \text{for any}\ n\in\mathbb{Z}. \]
By the long exact sequence of pair $(X, E(P_e))$:
\[\cdots\overset{\partial_{n+1}}{\longrightarrow} H_n(E(P_e)) \overset{s_{n}}{\longrightarrow} H_n(X) \overset{t_{n}}{\longrightarrow} H_n(X, E(P_e)) \overset{\partial_{n}}{\longrightarrow}\cdots,\]
we have the homology groups above. 
\end{proof}

\begin{lem}
\label{i21}
Let $t_n:H_n(X)\rightarrow H_n(X,E(P_e)), i_{n1}:H_n(\partial P)\rightarrow H_n(E(P_e)), $\\
$i_{n2}:H_n(\partial P)\rightarrow H_n(P)$ be homomorphisms induced by inclusion maps. 
If $t_2=0$, $i_{11}([l])=0$ and $i_{21}(H_2(\partial P))$ is included in 
$\mathbb{Z}[B_e]\oplus \mathbb{Z}[S]$, then we have $i_{21}([B])=\pm p[B_e]$ and $i_{21}([S])=0$. 
\end{lem}

\begin{proof}
By the definitions of $E_0, E_1, E_2, E_3, E_4$ and $E_5$ in the proof of Proposition \ref{recipe of pochette surgery}, we obtain 
\[E_{i*}([m])=
\begin{cases}
   [m]&(i=0, 2, 4), \\
   [m]+[l]&(i=1), \\
   [m]-[l]&(i=3), \\
   -[m]&(i=5),
\end{cases}
E_{i*}([l])=
\begin{cases}
   [l]&(i=0, 1, 3, 5), \\
   [m]+[l]&(i=2), \\
   -[m]+[l]&(i=4),
\end{cases}
\]
\[E_{i*}([B])=
\begin{cases}
   [B]&(i=0, 2, 4), \\
   [B]\mp[S]&(i=1), \\
   [B]\pm[S]&(i=3), \\
   \pm[B]&(i=5),
\end{cases}
E_{i*}([S])=
\begin{cases}
   [S]&(i=0, 1, 3), \\
   \mp[B]+[S]&(i=2), \\
   \pm[B]+[S]&(i=4), \\
   \mp[S]&(i=5)
\end{cases}
\]
(double-sign corresponds).
Then, there exist some integers $r, s$ such that $g_{p/q, \varepsilon*}([m])$\\
$=$$p[m_e]+q[l_e]$, $g_{p/q, \varepsilon*}([l])$$=$ $r[m_e]+s[l_e]$, $g_{p/q, \varepsilon\ast}([B])$$=$$\pm p[B_e]\pm q[S_e]$ and $g_{p/q, \varepsilon*}([S])$\\
$=$$\pm r[B_e]\pm s[S_e]$. 
Let $i_{\partial E(P_e)}:\partial E(P_e)\rightarrow E(P_e)$ be the inclusion map.
Then we have $i_{n1}=(i_{\partial E(P_e)}\circ g_{p/q,\varepsilon})_{n*}$ for any $n\in\mathbb{Z}$.
Here $f_{n*}:H_n(A)\rightarrow H_n(B)$ is the $n$-th induced homomorphism on homology of a continuous map $f:A\rightarrow B$.
If $t_2=0$, by Lemma \ref{homology EPe}, $H_1(E(P_e))={\mathbb Z}[m_e]$ and $H_2(E(P_e))={\mathbb Z}[B_e]\oplus H_2(X)$.
If $i_{11}([l])=0$ and $i_{21}(H_2(\partial P))$ is included in 
$\mathbb{Z}[B_e]\oplus \mathbb{Z}[S]$, then we have $i_{21}([S])=0$.
By $i_{11}([m])=p[m_e]$, we also have $i_{21}([B])=\pm p[B_e]$. 
Therefore, we obtain the desired result above. 
\end{proof}

For a simply-connected closed $4$-manifold $X$, we give a necessary condition for a pochette surgery of $X$ to have the same homology as $X$.
\begin{prop}
\label{the homology of pochette surgery}
Let $t_n:H_n(X)\rightarrow H_n(X,E(P_e)), i_{n1}:H_n(\partial P)\rightarrow H_n(E(P_e))$, \\
$i_{n2}:H_n(\partial P)\rightarrow H_n(P)$ be homomorphisms induced by inclusion maps. 
If $t_2=0$, $i_{11}([l])=0$ and $i_{21}(H_2(\partial P))$ is included in $\mathbb{Z}[B_e]\oplus \mathbb{Z}[S]$, then the homology groups of the pochette surgery $X(e,p/q,\varepsilon)$ are calculated as follows:
$$H_n(X(e, p/q, \varepsilon))=
\begin{cases}
   \mathbb{Z}&(n=0, 4), \\
   \mathbb{Z}_{p}&(n=1). 
\end{cases}$$
Moreover, if $|p|$ is equal to $1$, then $X(e,p/q,\varepsilon)$ has the same homology groups as $X$.  
\end{prop}

\begin{proof}
Since $X$ is connected and oriented, $H_n(X(e, p/q, \varepsilon))\cong{\mathbb Z}$ for any $n=0, 4$.  
We compute $H_1(X(e, p/q, \varepsilon))$ here. 
By Lemma \ref{homology EPe} and the Mayer-Vietoris sequence 
\[\cdots\overset{\partial_{n+1}}{\longrightarrow}H_n(\partial P)\overset{i_{n1}\oplus i_{n2}}{\longrightarrow} H_n(E(P_e))\oplus H_n(P)\overset{j_n}{\longrightarrow} H_n(X(e, p/q, \varepsilon))\overset{\partial_{n}}{\longrightarrow}\cdots, \]
we obtain the following: 
\[
  \begin{CD}
      @>>>  0 @>>>  H_3(X(e, p/q, \epsilon))  @>{\partial_{3}}>> {\mathbb Z}[B]\oplus{\mathbb Z}[S] \\
      @>{i_{21}\oplus i_{22}}>> ({\mathbb Z}[B_e]\oplus H_2(X))\oplus{\mathbb Z}[S] @>{j_2}>> H_2(X(e, p/q, \epsilon)) @>{\partial_{2}}>> {\mathbb Z}[m]\oplus{\mathbb Z}[l] \\
      @>{i_{11}\oplus i_{12}}>> {\mathbb Z}[m_e]\oplus{\mathbb Z}[l] @>{j_1}>> H_1(X(e, p/q, \epsilon)) @>{\partial_{1}}>> H_0(\partial P). 
  \end{CD}
\]
Then, we have $i_{11}([m])=p[m_e], i_{12}([m])=0, i_{12}([l])=[l]$ and $i_{22}([B])=0, i_{22}([S])$ \\
$=[S]. $
If $i_{11}([l])=0$, then we have 
\[\mathrm{Ker}\, \partial_{1}=\mathrm{Im}\, j_{1}\cong{\mathbb Z}[m_e]\oplus{\mathbb Z}[l]/\mathrm{Im}\, (i_{11}\oplus i_{12}) \cong\mathbb{Z}[m_e]/p\mathbb{Z}[m_e]\cong\mathbb{Z}_{p}\]
and $\mathrm{Im}\, \partial_{1}=0. $
Thus $H_1(X(e, p/q, \varepsilon))\cong\mathbb{Z}_{p}. $
If $|p|=1$, then we have 
$H_n(X(e, 1/q, \varepsilon))$\\
$\cong H_n(X)$ for any $n=0, 1, 4.$
By Lemma \ref{i21}, we have $i_{21}([B])=\pm p[B_e], i_{21}([S])=0. $
Therefore, we have $\mathrm{Ker}\, \partial_{3}=0, \mathrm{Im}\, \partial_{3}=\mathrm{Ker}\, (i_{21}\oplus i_{22})=0, $ 
and
\[\mathrm{Ker}\, \partial_{2}=\mathrm{Im}\, j_{2}\cong({\mathbb Z}[B_e]\oplus H_2(X))\oplus{\mathbb Z}[S]/\mathrm{Im}\, (i_{21}\oplus i_{22})=H_2(X), \]
$\mathrm{Im}\, \partial_{2}=\mathrm{Ker}\, (i_{11}\oplus i_{12})=0$. 
Then we have $H_n(X(e, 1/q, \varepsilon))\cong H_n(X)$ for any $n = 2, 3.$
\end{proof}

\begin{rmk}
\rm Proposition \ref{the homology of pochette surgery} is a generalization of Theorem \ref{homologyS4}. 
\end{rmk}

The next corollary follows from Proposition \ref{the homology of pochette surgery} and the Freedman theorem \cite{F, FQ}. 
\begin{cor}
\label{homeomorphic class}
If $t_2=0$, $i_{11}([l])=0$ and $i_{21}(H_2(\partial P))$ is included in $\mathbb{Z}[B_e]\oplus \mathbb{Z}[S]$, then $X(e, p/q, \varepsilon)$ is homeomorphic to $X$ if and only if $X(e, p/q, \varepsilon)$ is a simply connected 4-manifold and $|p|$ is equal to 1.
\end{cor}
\begin{proof}
If $X(e, p/q, \varepsilon)$ is homeomorphic to $X$, then $X(e, p/q, \varepsilon)$ has the same homology groups as $X$. 
By Proposition \ref{the homology of pochette surgery}, $X(e, p/q, \varepsilon)$ is a simply connected 4-manifold and $|p|=1$. 
Conversely, if $X(e, p/q, \varepsilon)$ is a simply connected 4-manifold and $|p|=1$, we obtain a natural isomorphism $H_2(X(e, p/q, \varepsilon))\cong H_2(X)$ by the proof of Proposition \ref{the homology of pochette surgery}. 
Hence, $Q_{X(e, p/q, \varepsilon)}\cong Q_X$. 
Here, $Q_Y$ is the intersection form of the 4-manifold $Y$. 
Since $X(e, p/q, \varepsilon)$ and $X$ are simply connected 4-dimensional closed manifolds with differential structures, $X(e, p/q, \varepsilon)\times\mathbb{R}$ and $X\times\mathbb{R}$ have differential structures. 
Therefore, we obtain $ks(X(e, p/q, \varepsilon))=0=ks(X)$. 
Here, $ks(Y)$ is the Kirby-Siebenmann invariant of $Y$. 
By the Freedman theorem, $X(e, p/q, \varepsilon)$ is homeomorphic to $X$. 
Therefore, we obtain the desired result above. 
\end{proof}


\section{Proofs of main theorems}
\label{Proofs of main theorems}
Canceling the $1$-handle/$2$-handle pairs and the $2$-handle/$3$-handle pair in the handle diagrams depicted in Figures \ref{Theorem 1.1-1} and \ref{Theorem 1.2-1}, we obtain the standard handle diagram of the $4$-sphere which consists of a $0$-handle and a $4$-handle.
Therefore both of the handle diagrams depicted in Figures \ref{Theorem 1.1-1} and \ref{Theorem 1.2-1} are those of the $4$-sphere. 
By Proposition \ref{recipe of pochette surgery} and Remark \ref{a generalization of prop}, the handle diagrams depicted in Figures \ref{Theorem 1.1-2} and \ref{Theorem 1.2-2} are those of the manifold $S^4(e_{k, n}, 1/q, \varepsilon)$ and the manifold $S^4(e_{m, n}, 1/q, \varepsilon)$, respectively. 
Let $H(n)$ be the union of $n$ $3$-handles and a $4$-handle. 
\vskip\baselineskip
\noindent
{\it Proof of Theorem 1.1.}\ We remove the $3$-handle and the $4$-handle in the handle diagram depicted in Figure \ref{Theorem 1.1-2}. 
Taking the double of the obtained handle diagram and removing all the $1$-handles and $2$-handles that has existed since the handle diagram depicted in Figure \ref{Theorem 1.1-2}, we obtain the handle diagram depicted in Figure \ref{Theorem 1.1proof1}. 
By several handle slides on $0$-framed meridians and $\langle0\rangle$-framed knots in Figure \ref{Theorem 1.1proof1}, $4k-7\ 0$-framed knots can be changed to $4k-7$ $0$-framed unknots. 
Canceling the $4k-7$ $2$-handle/$3$-handle pairs, we obtain the handle diagram depicted in Figure \ref{Theorem 1.1proof2}. 
By the handle slides in Figure \ref{Theorem 1.1proof2}, we obtain the handle diagram depicted in Figure \ref{Theorem 1.1proof3}. 
By several handle slides on the $2$ $\langle0\rangle$-framed meridians in Figure \ref{Theorem 1.1proof3}, we obtain the handle diagram depicted in Figure \ref{Theorem 1.1proof4}. 
We can cancel the $2$ Hopf links which consists of $2$ $\langle\cdot\rangle$-framed knots in Figure \ref{Theorem 1.1proof4}. 
By the handle slides in Figure \ref{Theorem 1.1proof4} and several handle slides on the $3$ $\langle0\rangle$-framed meridians, we obtain the handle diagram depicted in Figure \ref{Theorem 1.1proof5}. 
By the handle slides in Figure \ref{Theorem 1.1proof5}, we obtain a Hopf link and a $0$-framed meridian. 
By several handle slides on $0$-framed meridians and $\langle0\rangle$-framed knots in Figure \ref{Theorem 1.1proof5}, we obtain the handle diagram depicted in Figure \ref{Theorem 1.1proof6}. 
Repeating the handle calculus in Figure \ref{Theorem 1.1proof2}--\ref{Theorem 1.1proof6} in the same way, we obtain the handle diagram depicted in Figure \ref{Theorem 1.1proof7}. 
Here, $L=\sum_{i=3}^{k^2-1}n_i$.

By the handle slide in Figure \ref{Theorem 1.1proof7}, we obtain the handle diagram depicted in Figure \ref{Theorem 1.1proof8}.
By several handle slides on the moved $\langle0\rangle$-framed knot, we obtain the handle diagram depicted in Figure \ref{Theorem 1.1proof9}.
Canceling the 2 $2$-handle/$3$-handle pairs and the 3 Hopf links, we obtain the handle diagrams depicted in Figure \ref{Theorem 1.1proof10}, \ref{Theorem 1.1proof11} and \ref{Theorem 1.1proof12} in order.
Changing the $\langle0\rangle$-framed knot in Figure \ref{Theorem 1.1proof12} to a dotted circle, the upside down of the handle diagram depicted in Figure \ref{Theorem 1.1-2} is completed.
Canceling the $1$-handle/$2$-handle pair, we obtain the standard handle diagram of the $4$-sphere which consists of a $0$-handle and a $4$-hanlde. 

Therefore, $S^4(e_{k, n}, 1/q, \varepsilon)$ is diffeomorphic to $S^4$.
\qed
\vskip\baselineskip
\noindent
{\it Proof of Theorem 1.2.}\ By the handle slides in Figure \ref{Theorem 1.2proof1}, several handle slides on dotted circles and canceling the $s(t-1)$ $1$-handle/$2$-handle pairs, we obtain the handle diagram depicted in Figure \ref{Theorem 1.2proof2}. 
Here, $x_i=\sum_{j=1}^{t}n_{s(j-1)+i}$ for any $i=1,\ldots,s$.
By the handle slides in Figure \ref{Theorem 1.2proof2}--\ref{Theorem 1.2proof3}, we obtain the handle diagram depicted in Figure \ref{Theorem 1.2proof4}. 
Here, $a,b_1$ are some integers.
Repeating the method of the handle calculus in Figure \ref{Theorem 1.2proof3}--\ref{Theorem 1.2proof4}, we obtain the handle diagram depicted in Figure \ref{Theorem 1.2proof5}. 
Here, $b_2,\ldots,b_{s-2}$ are some integers.

By the handle slide in Figure \ref{Theorem 1.2proof6} or \ref{Theorem 1.2proof8}, we obtain the handle diagram depicted in Figure \ref{Theorem 1.2proof10}.
By the handle slide in Figure \ref{Theorem 1.2proof10}, we obtain the handle diagram depicted in Figure \ref{Theorem 1.2proof11}. 
By several handle slides on the 0-framed meridian, the handle diagram depicted in Figure \ref{Theorem 1.2proof11} can be changed to that depicted in Figure \ref{Theorem 1.2proof12}. 
By the handle calculus in Figure \ref{Theorem 1.2proof10}--\ref{Theorem 1.2proof11} or \ref{Theorem 1.2proof12} $|m_s|$ times, we obtain the handle diagram depicted in Figure \ref{Theorem 1.2proof13}. 
By the handle slide in Figure \ref{Theorem 1.2proof13} and the method of the handle calculus in Figure \ref{Theorem 1.2proof10}--\ref{Theorem 1.2proof11} or \ref{Theorem 1.2proof12} $|m_{s-1}+m_s|$ times, we obtain the handle diagram depicted in Figure \ref{Theorem 1.2proof14}. 
Here, $c=m_{s-1}+m_s$.
By repeating the method of the handle calculus in Figure \ref{Theorem 1.2proof10}--\ref{Theorem 1.2proof14} and $\sum_{i=1}^{s}m_i=0$, we can the $(x_1+\varepsilon)$-framed knot in Figure \ref{Theorem 1.2proof5} away from the dotted circles entwined with the $n_{st+1}$-framed knot. 
Thus, we obtain the handle diagram depicted in Figure \ref{Theorem 1.2proof15}. 

By the handle slide in Figure \ref{Theorem 1.2proof6} (Figure \ref{Theorem 1.2proof8}) and the handle calculus in Figure \ref{Theorem 1.2proof10}--\ref{Theorem 1.2proof15}, we obtain the handle diagram depicted in Figure \ref{Theorem 1.2proof7} (Figure \ref{Theorem 1.2proof9}) (double sign corresponds).
Therefore, we have the handle diagram depicted in Figure \ref{Theorem 1.2proof5} with $q=0$ by the handle slide in Figure \ref{Theorem 1.2proof6} or \ref{Theorem 1.2proof8} $|q|$ times. By canceling the $s+1$ $1$-handle/$2$-handle pairs and the $2$-handle/$3$-handle pair in the handle diagrams depicted in Figure \ref{Theorem 1.2proof5} with $q=0$, we obtain the standard handle diagram of the $4$-sphere which consists of a $0$-handle and a $4$-handle.

Therefore, $S^4(e_{m, n}, 1/q, \varepsilon)$ is diffeomorphic to $S^4$.
\qed\\

We end this section by raising a conjecture about pochette surgery of the $4$-sphere. 
\begin{conj}
\rm If the pochette surgery $S^4(e, p/q, \varepsilon)$ is homotopy equivalent to the 4-sphere $S^4$, then $S^4(e, p/q, 0)$ is diffeomorphic to $S^4$ and $S^4(e, p/q, 1)$ is the Gluck surgery along $S_e$. 
Especially when $S_e$ is a twist spun $2$-knot or a $0$-slice $2$-knot, $S^4(e, p/q, \varepsilon)$ is diffeomorphic to $S^4$ for any $\varepsilon\in\{0,1\}$. 
\end{conj}

\begin{figure}[htbp]
  \begin{minipage}[b]{0.45\linewidth}
    \centering
     \begin{overpic}[scale=0.5]
     {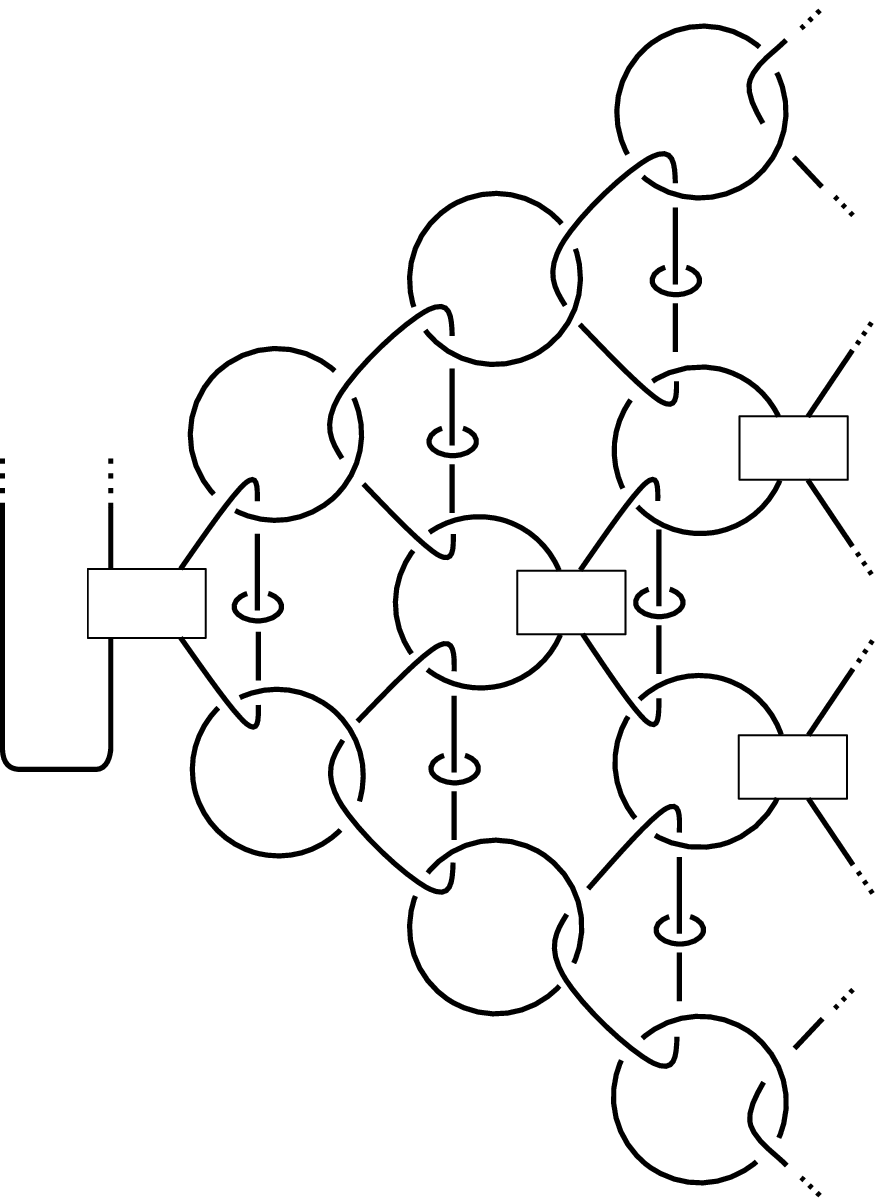}
     \put(1,30){$\langle0\rangle$}
     \put(19.5,22){$\langle0\rangle$}
     \put(19.5,74){$\langle0\rangle$}
     \put(37.5,10){$\langle0\rangle$}
     \put(34,48.5){$\langle0\rangle$}
     \put(37.5,87){$\langle0\rangle$}
     \put(55,-5){$\langle0\rangle$}
     \put(44,34.5){$\langle0\rangle$}
     \put(44,62){$\langle0\rangle$}
     \put(55,101){$\langle0\rangle$}

     \put(9,48){$\pm1$}
     \put(44.5,48){$\pm1$}
     \put(63,34){$\pm1$}
     \put(63,61){$\pm1$}

     \put(24.5,48){$0$}
     \put(40.5,34.5){$0$}
     \put(40.5,62){$0$}
     \put(59.5,21){$0$}
     \put(58,48){$0$}
     \put(59.5,75){$0$}
     \end{overpic}
    \caption{}
    \label{Theorem 1.1proof1}
  \end{minipage}
  \begin{minipage}[b]{0.45\linewidth}
    \centering
     \begin{overpic}[scale=0.5]
     {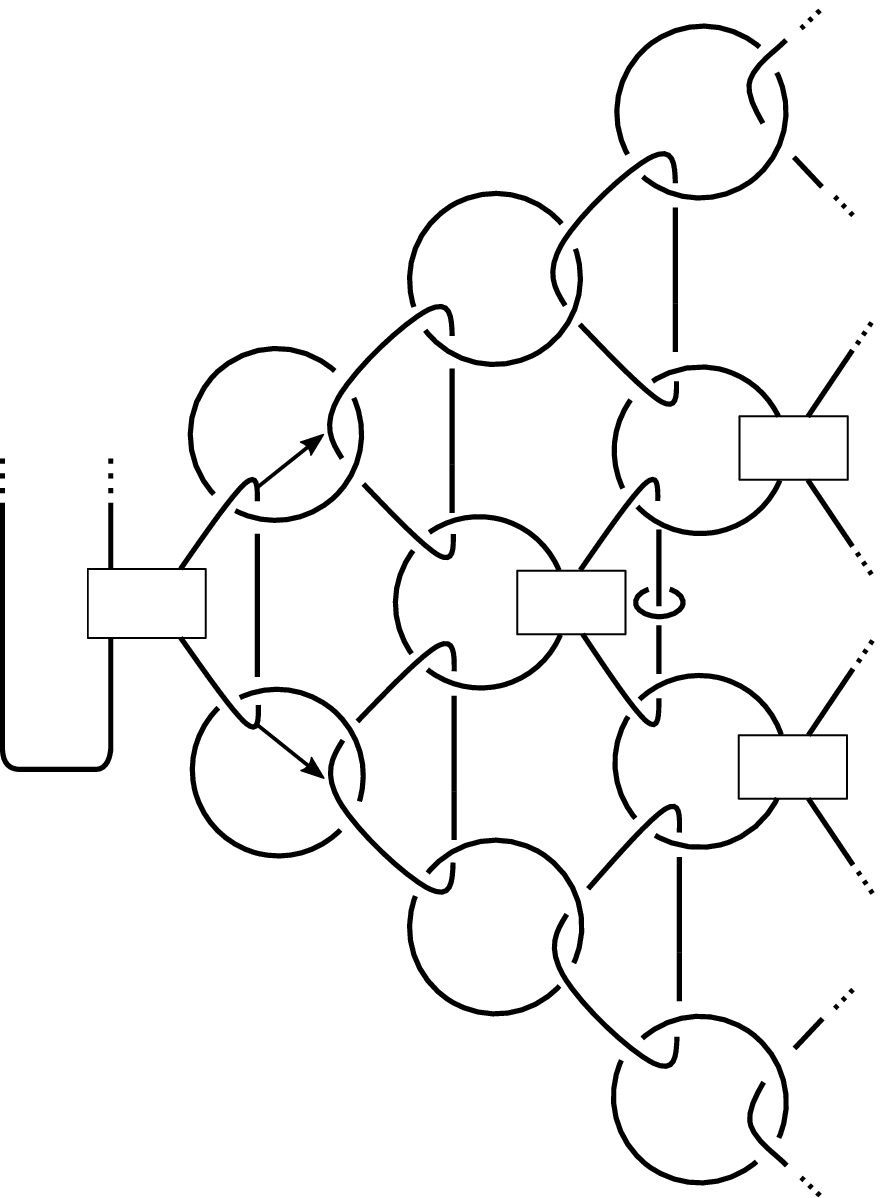}
     \put(1,30){$\langle0\rangle$}
     \put(19.5,22){$\langle0\rangle$}
     \put(19.5,74){$\langle0\rangle$}
     \put(37.5,10){$\langle0\rangle$}
     \put(34,48.5){$\langle0\rangle$}
     \put(37.5,87){$\langle0\rangle$}
     \put(55,-5){$\langle0\rangle$}
     \put(44,34.5){$\langle0\rangle$}
     \put(44,62){$\langle0\rangle$}
     \put(55,101){$\langle0\rangle$}

     \put(9,48){$\pm1$}
     \put(44.5,48){$\pm1$}
     \put(63,34){$\pm1$}
     \put(63,61){$\pm1$}

     \put(58,48){$0$}
     \end{overpic}
    \caption{}
    \label{Theorem 1.1proof2}
  \end{minipage}
\end{figure}

\begin{figure}[htbp]
  \begin{minipage}[b]{0.45\linewidth}
    \centering
     \begin{overpic}[scale=0.5]
     {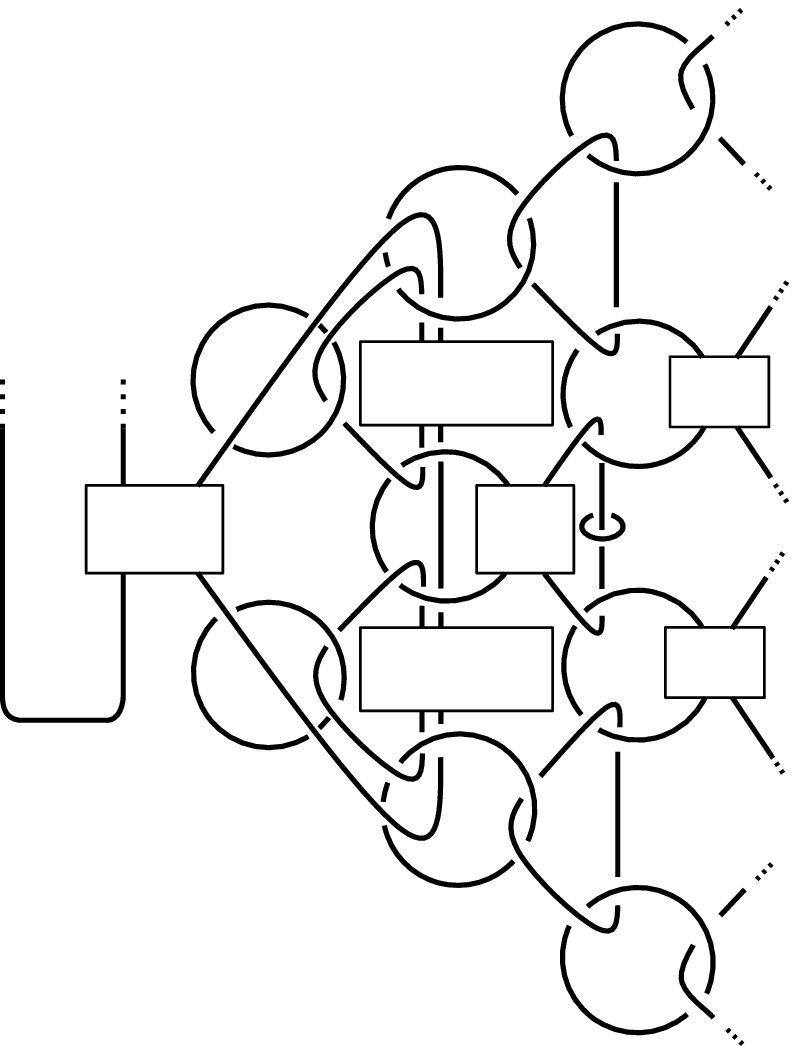}
     \put(2,25){$\langle0\rangle$}
     \put(20,22){$\langle0\rangle$}
     \put(20,74){$\langle0\rangle$}
     \put(39,9){$\langle0\rangle$}
     \put(27,48.5){$\langle0\rangle$}
     \put(39,88){$\langle0\rangle$}
     \put(56,-5){$\langle0\rangle$}
     \put(60,23){$\langle0\rangle$}
     \put(60,73){$\langle0\rangle$}
     \put(56,101){$\langle0\rangle$}

     \put(11,48){$\pm1$}
     \put(46.5,48){$\pm1$}
     \put(65,34.5){$\pm1$}
     \put(65,61){$\pm1$}

     \put(61,48){$0$}

     \put(36,35){$n_{k^2-2}$}
     \put(36,62.5){$n_{k^2-3}$}
     \end{overpic}
    \caption{}
    \label{Theorem 1.1proof3}
  \end{minipage}
  \begin{minipage}[b]{0.45\linewidth}
    \centering
     \begin{overpic}[scale=0.5]
     {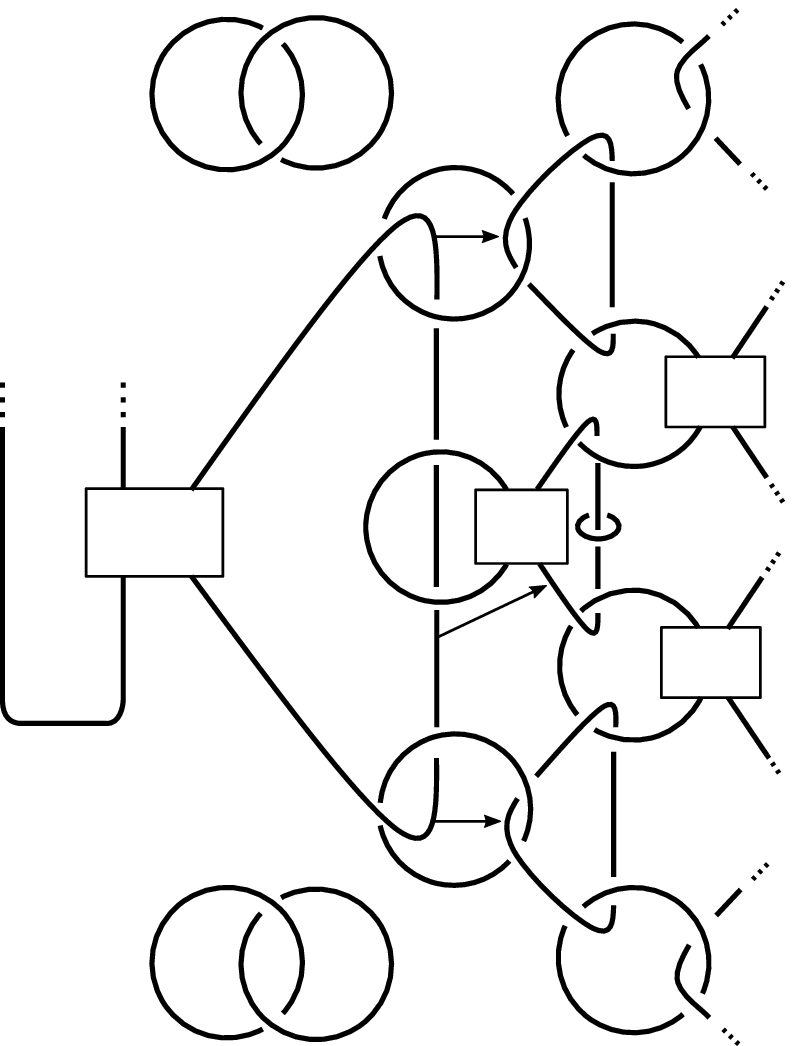}
     \put(2,25){$\langle0\rangle$}
     \put(6,12){$\langle0\rangle$}
     \put(6,83){$\langle0\rangle$}
     \put(39,9){$\langle0\rangle$}
     \put(27,48.5){$\langle0\rangle$}
     \put(39,88){$\langle0\rangle$}
     \put(56,-5){$\langle0\rangle$}
     \put(60,23){$\langle0\rangle$}
     \put(60,73){$\langle0\rangle$}
     \put(56,101){$\langle0\rangle$}

     \put(11,48){$\pm1$}
     \put(46.5,48){$\pm1$}
     \put(65,34.5){$\pm1$}
     \put(65,61){$\pm1$}

     \put(61,48){$0$}
     \end{overpic}
    \caption{}
    \label{Theorem 1.1proof4}
  \end{minipage}
\end{figure}

\begin{figure}[htbp]
  \begin{minipage}[b]{0.45\linewidth}
    \centering
     \begin{overpic}[scale=0.5]
     {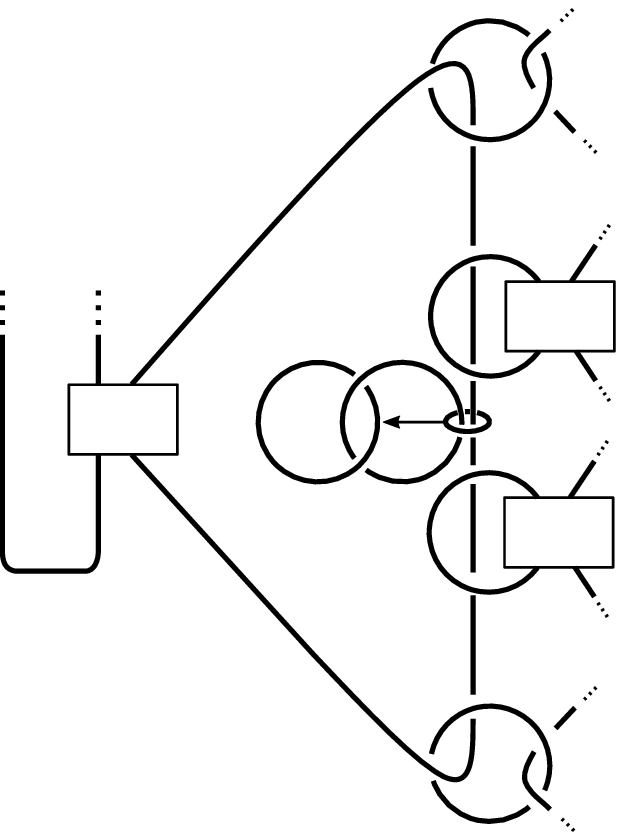}
     \put(1,23){$\langle0\rangle$}
     \put(26,58){$\langle0\rangle$}
     \put(42,0){$\langle0\rangle$}
     \put(40,33){$\langle0\rangle$}
     \put(40,61){$\langle0\rangle$}
     \put(42,96){$\langle0\rangle$}

     \put(10,47.5){$\pm1$}
     \put(63,34){$\pm1$}
     \put(63,60){$\pm1$}

     \put(60,46){$0$}
     \end{overpic}
    \caption{}
    \label{Theorem 1.1proof5}
  \end{minipage}
  \begin{minipage}[b]{0.45\linewidth}
    \centering
     \begin{overpic}[scale=0.5]
     {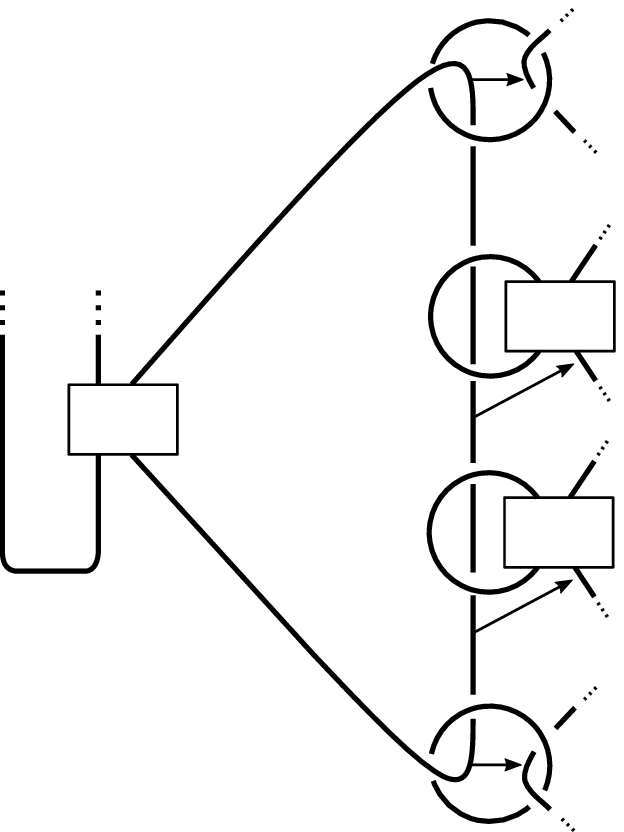}
     \put(1,23){$\langle0\rangle$}
     \put(42,0){$\langle0\rangle$}
     \put(40,33){$\langle0\rangle$}
     \put(40,61){$\langle0\rangle$}
     \put(42,96){$\langle0\rangle$}

     \put(10,47.5){$\pm1$}
     \put(63,34){$\pm1$}
     \put(63,60){$\pm1$}

     \end{overpic}
    \caption{}
    \label{Theorem 1.1proof6}
  \end{minipage}
\end{figure}

\begin{figure}[htbp]
  \begin{minipage}[b]{0.45\linewidth}
    \centering
     \begin{overpic}[scale=0.5]
     {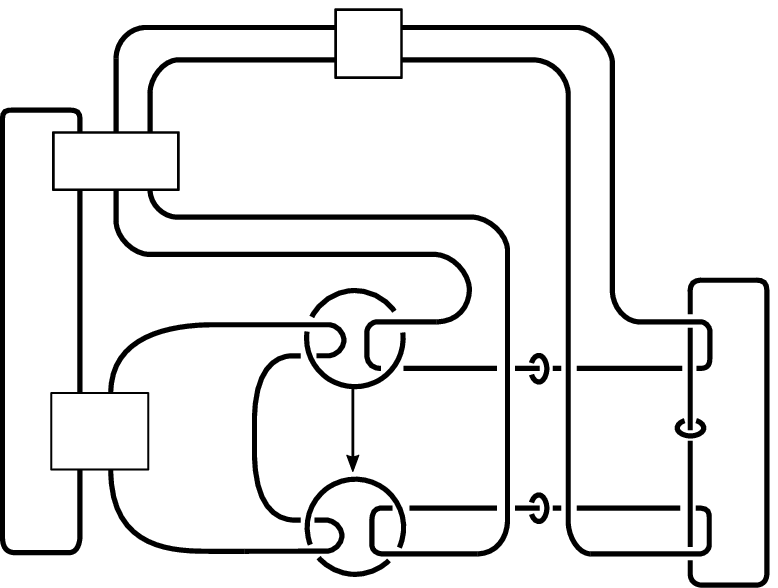}
     \put(13.5,54){$q$}
     \put(46,69){$\varepsilon$}

     \put(12,-4){$\langle L\rangle$}
     \put(81,60){$\langle n_1+\varepsilon\rangle$}
     \put(54,-4){$\langle n_2+\varepsilon\rangle$}

     \put(0,66){$\langle0\rangle$}
     \put(41,-6){$\langle0\rangle$}
     \put(49,20){$\langle0\rangle$}
     \put(89,44){$\langle0\rangle$}

     \put(8,18){$\pm1$}

     \put(68,33){$0$}
     \put(68,15){$0$}
     \put(82,18){$0$}

     \put(105,20){$\cup H(2)$}
     \end{overpic}
    \caption{}
    \label{Theorem 1.1proof7}
  \end{minipage}
  \begin{minipage}[b]{0.45\linewidth}
    \centering
     \begin{overpic}[scale=0.5]
     {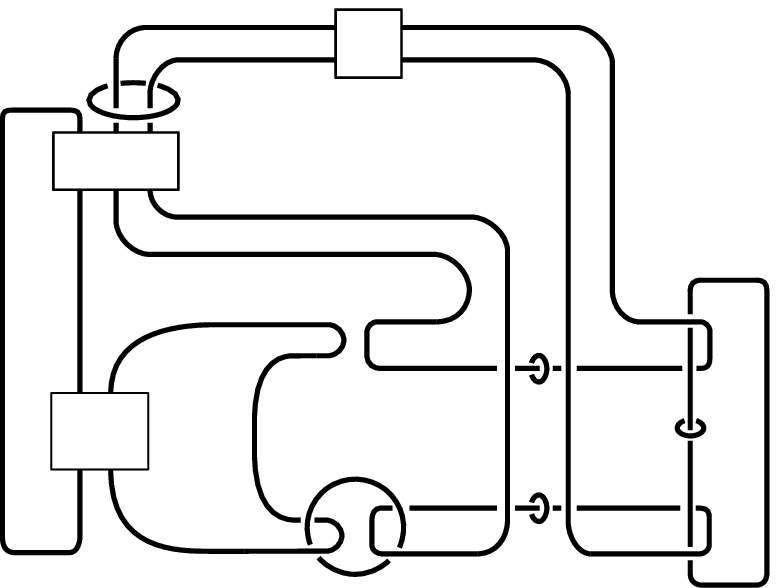}
     \put(13.5,54){$q$}
     \put(46,69){$\varepsilon$}

     \put(12,-4){$\langle L\rangle$}
     \put(81,60){$\langle n_1+\varepsilon\rangle$}
     \put(54,-4){$\langle n_2+\varepsilon\rangle$}

     \put(0,66){$\langle0\rangle$}
     \put(41,-6){$\langle0\rangle$}
     \put(24.5,60){$\langle0\rangle$}
     \put(89,44){$\langle0\rangle$}

     \put(8,18){$\pm1$}

     \put(68,33){$0$}
     \put(68,15){$0$}
     \put(82,18){$0$}

     \put(105,20){$\cup H(2)$}
     \end{overpic}
    \caption{}
    \label{Theorem 1.1proof8}
  \end{minipage}
\end{figure}

\begin{figure}[htbp]
  \begin{minipage}[b]{0.45\linewidth}
    \centering
     \begin{overpic}[scale=0.5]
     {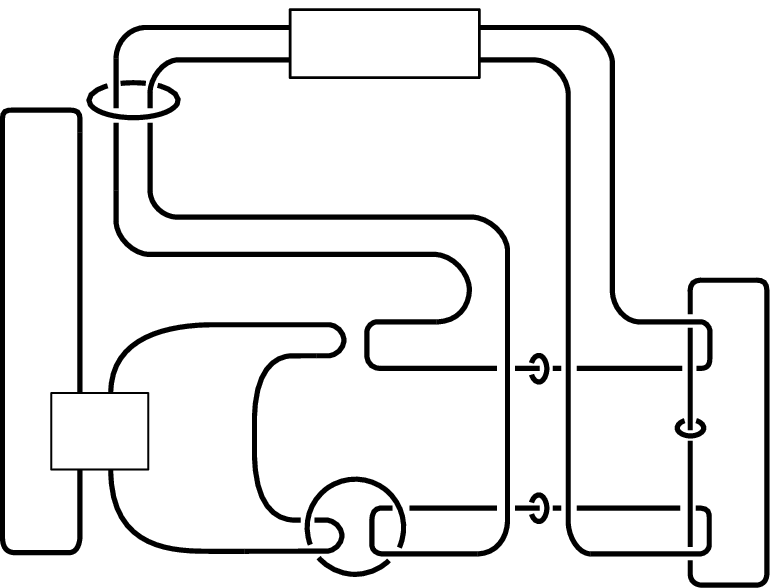}
     \put(40,69){$q+\varepsilon$}

     \put(12,-4){$\langle L\rangle$}
     \put(81,60){$\langle n_1+\varepsilon\rangle$}
     \put(54,-4){$\langle n_2+\varepsilon\rangle$}

     \put(0,66){$\langle0\rangle$}
     \put(41,-6){$\langle0\rangle$}
     \put(24.5,60){$\langle0\rangle$}
     \put(89,44){$\langle0\rangle$}

     \put(8,18){$\pm1$}

     \put(68,33){$0$}
     \put(68,15){$0$}
     \put(82,18){$0$}

     \put(105,20){$\cup H(2)$}
     \end{overpic}
    \caption{}
    \label{Theorem 1.1proof9}
  \end{minipage}
  \begin{minipage}[b]{0.45\linewidth}
    \centering
     \begin{overpic}[scale=0.5]
     {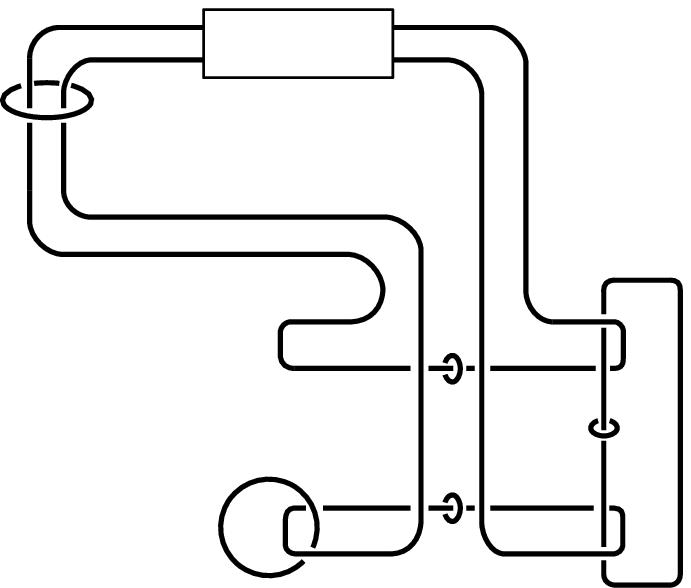}
     \put(33,78){$q+\varepsilon$}

     \put(80,66){$\langle n_1+\varepsilon\rangle$}
     \put(50,-5){$\langle n_2+\varepsilon\rangle$}

     \put(15,68){$\langle0\rangle$}
     \put(33,20){$\langle0\rangle$}
     \put(87,50){$\langle0\rangle$}

     \put(64,36){$0$}
     \put(64,16){$0$}
     \put(80,20){$0$}

     \put(105,20){$\cup H(2)$}
     \end{overpic}
    \caption{}
    \label{Theorem 1.1proof10}
  \end{minipage}
\end{figure}

\begin{figure}[htbp]
  \begin{minipage}[b]{0.45\linewidth}
    \centering
     \begin{overpic}[scale=0.5]
     {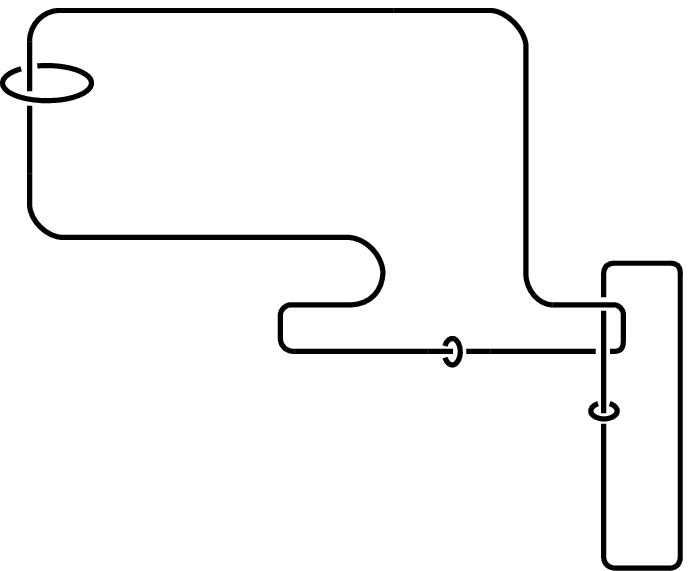}
     \put(64,38){$0$}
     \put(80,20){$0$}

     \put(15,65){$\langle0\rangle$}
     \put(88,50){$\langle0\rangle$}

     \put(80,70){$\langle n_1+\varepsilon\rangle$}
     \put(105,20){$\cup H(1)$}
     \end{overpic}
    \caption{}
    \label{Theorem 1.1proof11}
  \end{minipage}
  \begin{minipage}[b]{0.45\linewidth}
    \centering
     \begin{overpic}[scale=0.5]
     {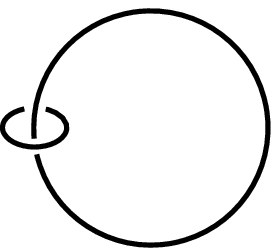}
     \put(-20,32){$0$}
     \put(40,100){$\langle0\rangle$}
     \put(120,32){$\cup H(0)$}
     \end{overpic}
    \caption{}
    \label{Theorem 1.1proof12}
  \end{minipage}
\end{figure}

\begin{figure}[htbp]
  \begin{center}
     \begin{overpic}[scale=0.5]
     {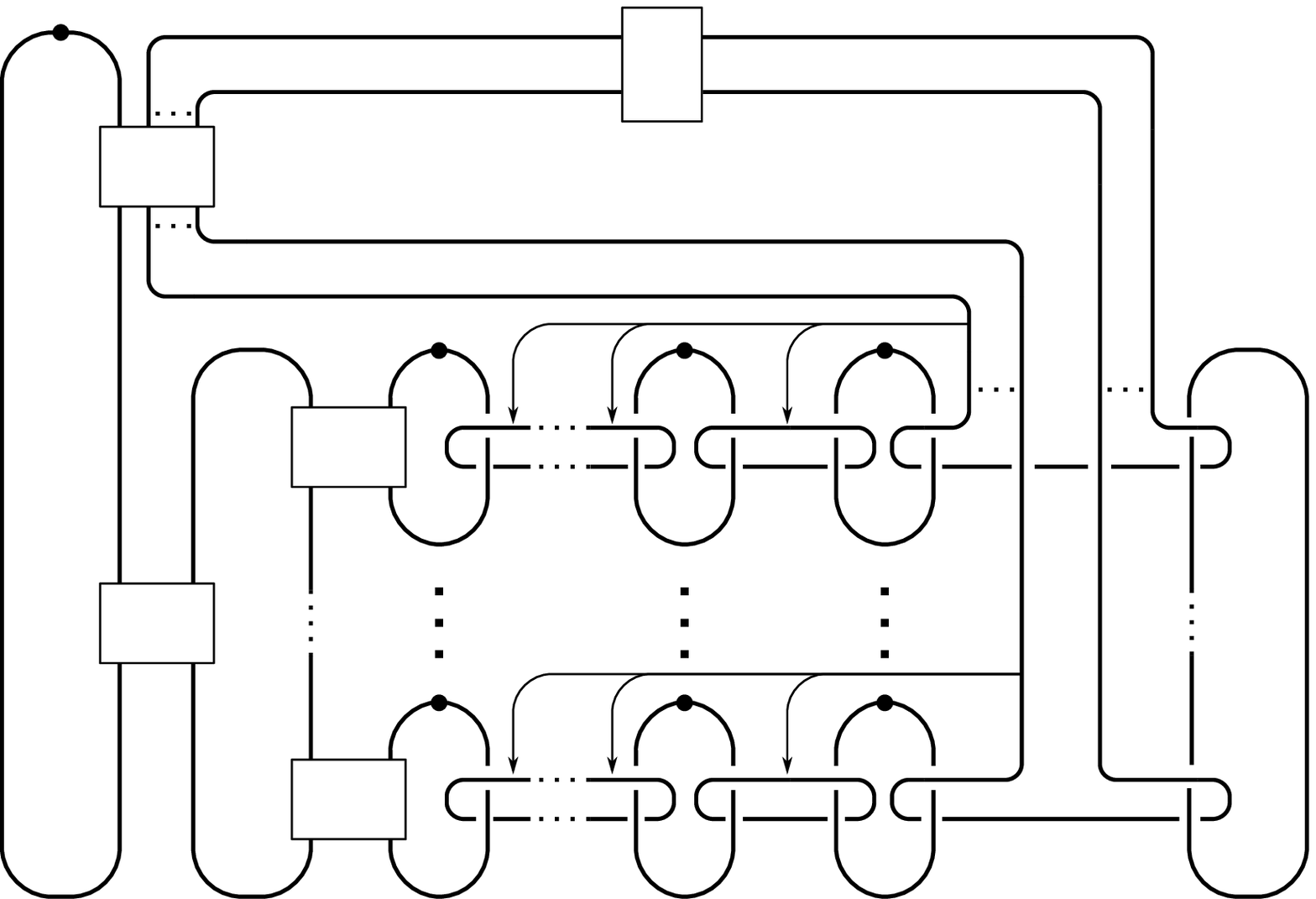}
     \put(94,43){$0$}

     \put(70,68){$n_1+\varepsilon$}
     \put(70,58){$n_s+\varepsilon$}

     \put(11,55){$q$}
     \put(49.5,62.5){$\varepsilon$}

     \put(24,33.5){$m_1$}
     \put(24,6.5){$m_s$}

     \put(9.5,19.5){$\pm1$}

     \put(105,24){$\cup$ 3-handle}
     \put(109.2,19){4-handle}
     \end{overpic}
    \caption{}
    \label{Theorem 1.2proof1}
  \end{center}
\end{figure}

\begin{figure}[htbp]
  \begin{center}
     \begin{overpic}[scale=0.5]
     {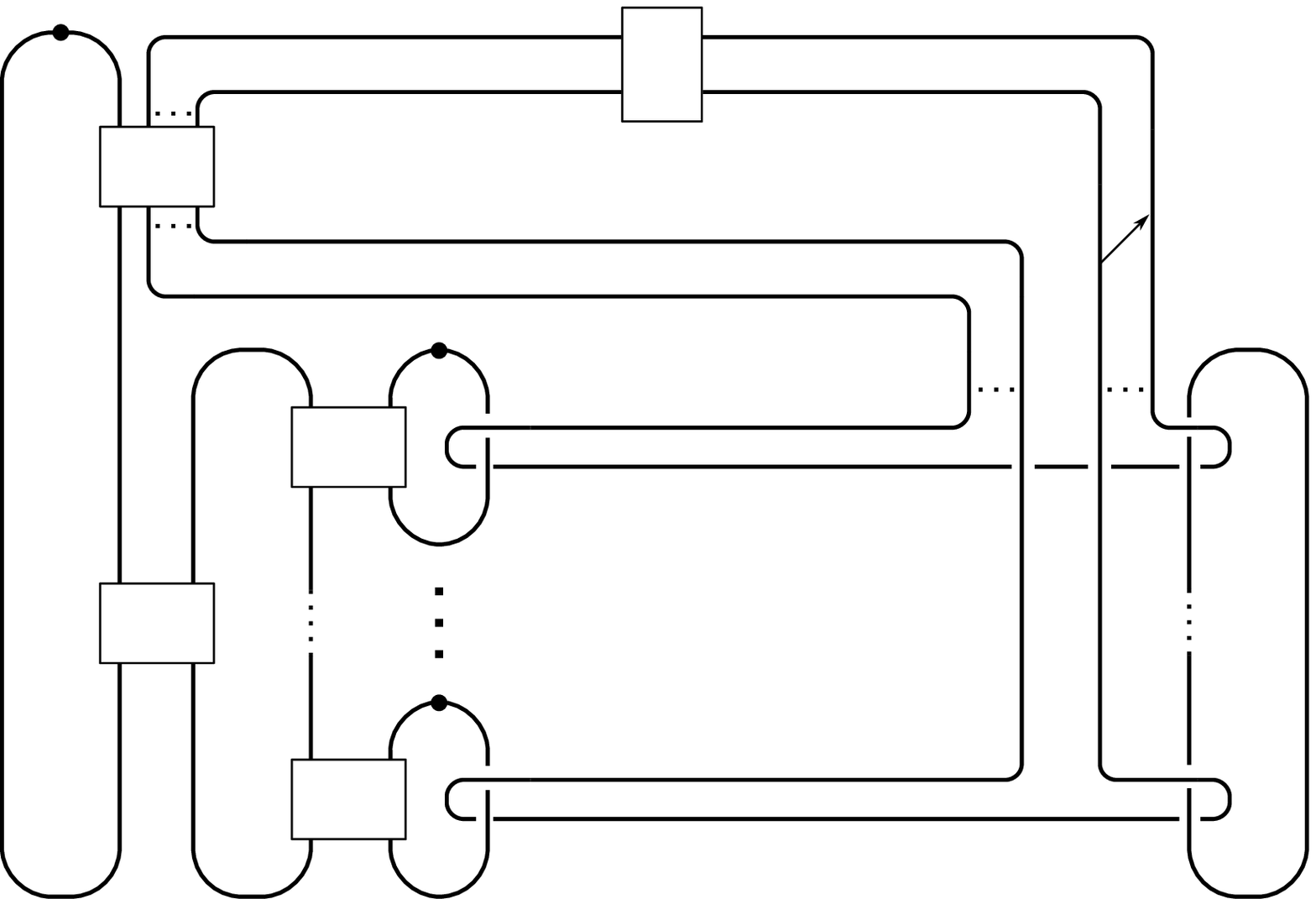}
     \put(94,43){$0$}

     \put(70,68){$x_1+\varepsilon$}
     \put(70,58){$x_s+\varepsilon$}

     \put(11,55){$q$}
     \put(49.5,62.5){$\varepsilon$}

     \put(24,33.5){$m_1$}
     \put(24,6.5){$m_s$}

     \put(9.5,19.5){$\pm1$}

     \put(105,24){$\cup$ 3-handle}
     \put(109.2,19){4-handle}
     \end{overpic}
    \caption{}
    \label{Theorem 1.2proof2}
  \end{center}
\end{figure}

\begin{figure}[htbp]
  \begin{minipage}[b]{0.45\linewidth}
    \centering
     \begin{overpic}[scale=0.5]
     {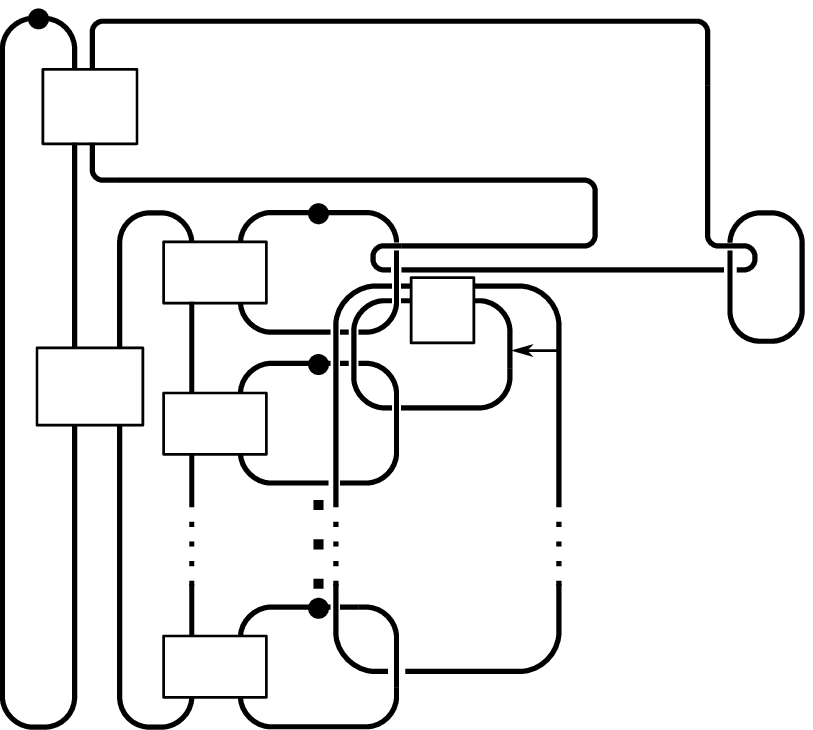}
     \put(30,91){$x_1+\varepsilon$}
     \put(9,76){$q$}

     \put(6,40){$\pm1$}

     \put(22,54.7){$m_1$}
     \put(22,36){$m_2$}
     \put(22,6){$m_s$}

     \put(52.5,50){$a$}

     \put(92.5,65){$0$}

     \put(72,26){$\cup$ 3-handle}
     \put(80.2,17){4-handle}
     \end{overpic}
    \caption{}
    \label{Theorem 1.2proof3}
  \end{minipage}
  \begin{minipage}[b]{0.45\linewidth}
    \centering
     \begin{overpic}[scale=0.5]
     {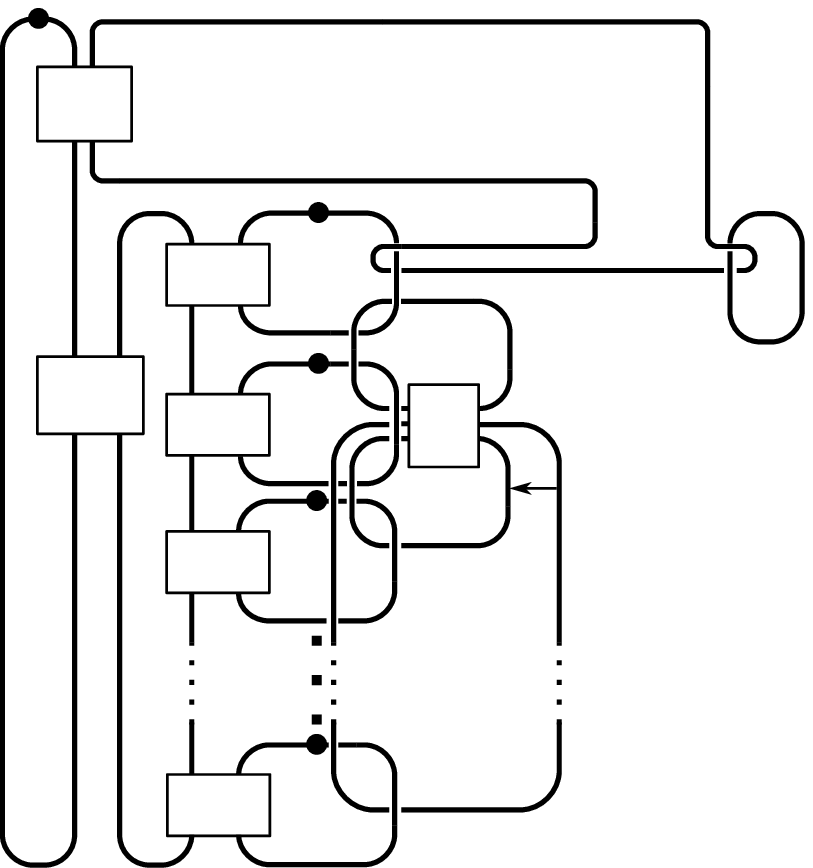}
     \put(30,101){$x_1+\varepsilon$}
     \put(8,87){$q$}

     \put(5.5,52){$\pm1$}

     \put(21,67){$m_1$}
     \put(21,50){$m_2$}
     \put(21,34){$m_3$}
     \put(21,6){$m_s$}

     \put(48.5,49){$b_1$}

     \put(87,78){$0$}

     \put(72,26){$\cup$ 3-handle}
     \put(80.2,17){4-handle}
     \end{overpic}
    \caption{}
    \label{Theorem 1.2proof4}
  \end{minipage}
\end{figure}

\begin{figure}[htbp]
  \begin{minipage}[b]{0.45\linewidth}
    \centering
     \begin{overpic}[scale=0.5]
     {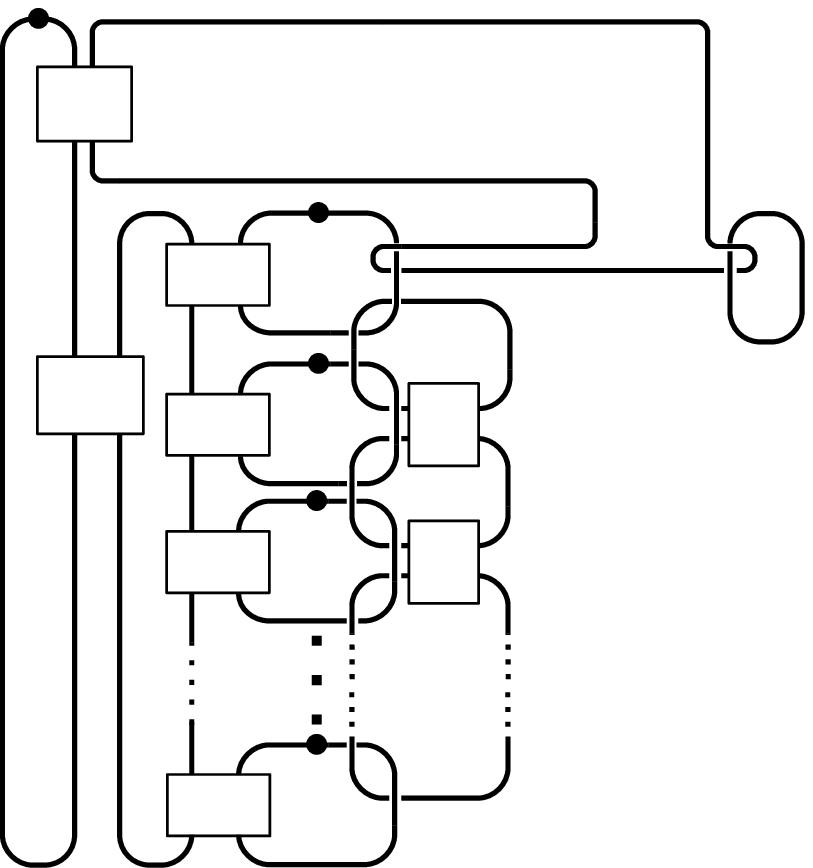}
     \put(30,101){$x_1+\varepsilon$}
     \put(8,87){$q$}

     \put(5.5,52){$\pm1$}

     \put(21,67){$m_1$}
     \put(21,50){$m_2$}
     \put(21,34){$m_3$}
     \put(21,6){$m_s$}

     \put(48.5,49){$b_1$}
     \put(48.5,33){$b_2$}

     \put(87,78){$0$}

     \put(72,26){$\cup$ 3-handle}
     \put(80.2,17){4-handle}
     \end{overpic}
    \caption{}
    \label{Theorem 1.2proof5}
  \end{minipage}
\end{figure}

\begin{figure}[htbp]
  \begin{minipage}[b]{0.45\linewidth}
    \centering
     \begin{overpic}[scale=0.5]
     {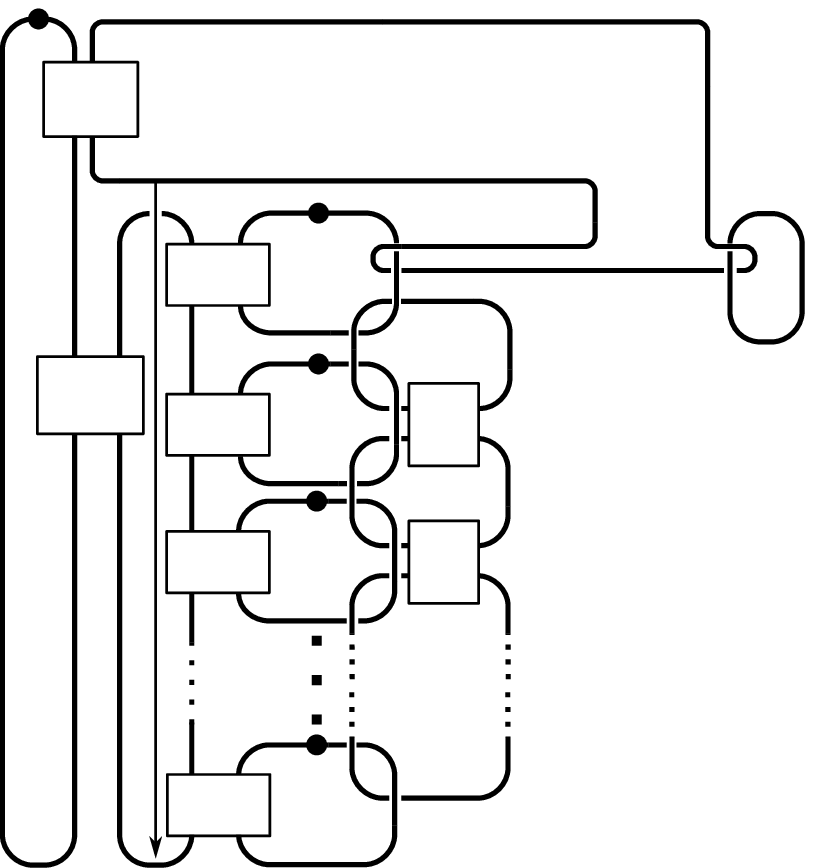}
     \put(30,101){$x_1+\varepsilon$}
     \put(8,87.5){$q$}

     \put(5.5,52){$\pm1$}

     \put(21,67){$m_1$}
     \put(21,50){$m_2$}
     \put(21,34){$m_3$}
     \put(21,6){$m_s$}

     \put(48.5,49){$b_1$}
     \put(48.5,33){$b_2$}

     \put(87,78){$0$}

     \put(72,26){$\cup$ 3-handle}
     \put(80.2,17){4-handle}
     \end{overpic}
    \caption{}
    \label{Theorem 1.2proof6}
  \end{minipage}
  \begin{minipage}[b]{0.45\linewidth}
    \centering
     \begin{overpic}[scale=0.5]
     {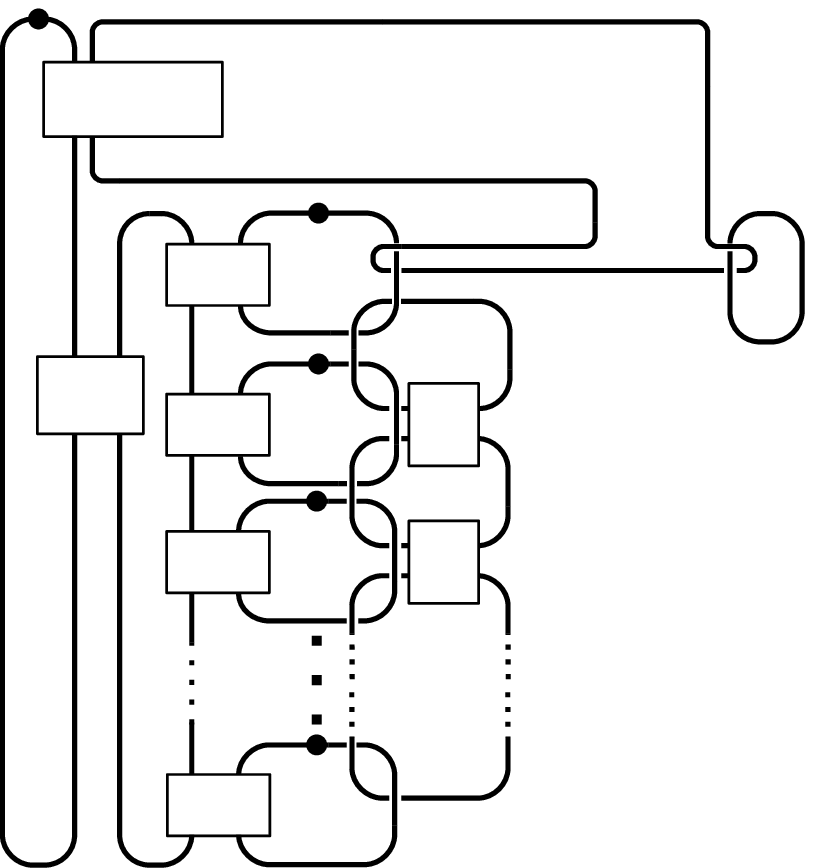}
     \put(30,101){$x_1+\varepsilon$}
     \put(7,87.5){$q\mp1$}

     \put(5.5,52){$\pm1$}

     \put(21,67){$m_1$}
     \put(21,50){$m_2$}
     \put(21,34){$m_3$}
     \put(21,6){$m_s$}

     \put(48.5,49){$b_1$}
     \put(48.5,33){$b_2$}

     \put(87,78){$0$}

     \put(72,26){$\cup$ 3-handle}
     \put(80.2,17){4-handle}
     \end{overpic}
    \caption{}
    \label{Theorem 1.2proof7}
  \end{minipage}
\end{figure}

\begin{figure}[htbp]
  \begin{minipage}[b]{0.45\linewidth}
    \centering
     \begin{overpic}[scale=0.5]
     {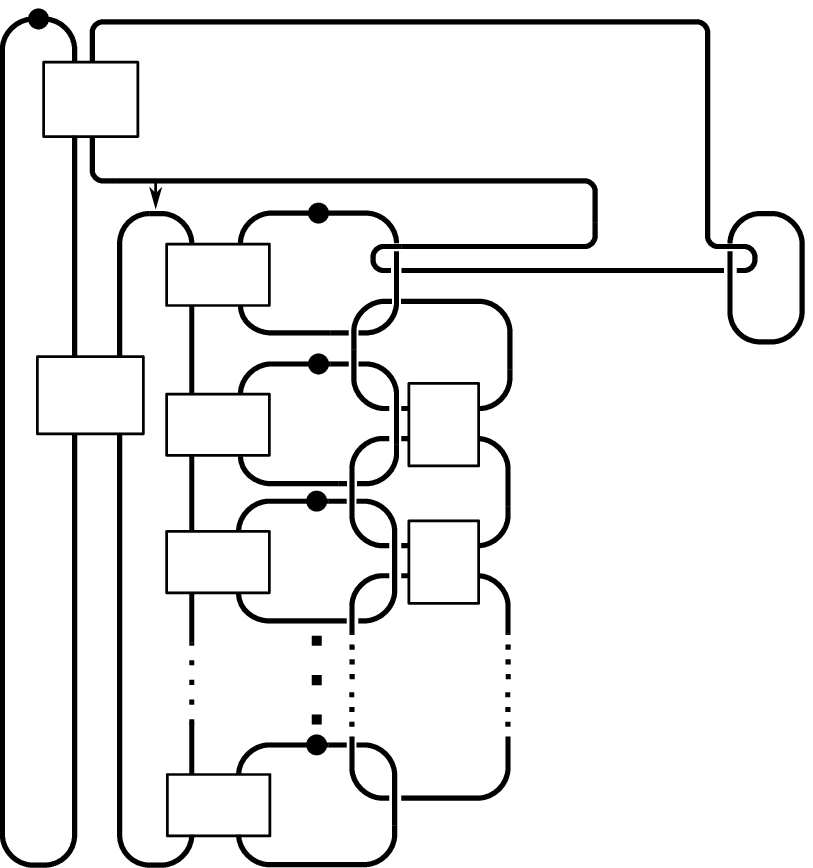}
     \put(30,101){$x_1+\varepsilon$}
     \put(8,87.5){$q$}

     \put(5.5,52){$\pm1$}

     \put(21,67){$m_1$}
     \put(21,50){$m_2$}
     \put(21,34){$m_3$}
     \put(21,6){$m_s$}

     \put(48.5,49){$b_1$}
     \put(48.5,33){$b_2$}

     \put(87,78){$0$}

     \put(72,26){$\cup$ 3-handle}
     \put(80.2,17){4-handle}
     \end{overpic}
    \caption{}
    \label{Theorem 1.2proof8}
  \end{minipage}
  \begin{minipage}[b]{0.45\linewidth}
    \centering
     \begin{overpic}[scale=0.5]
     {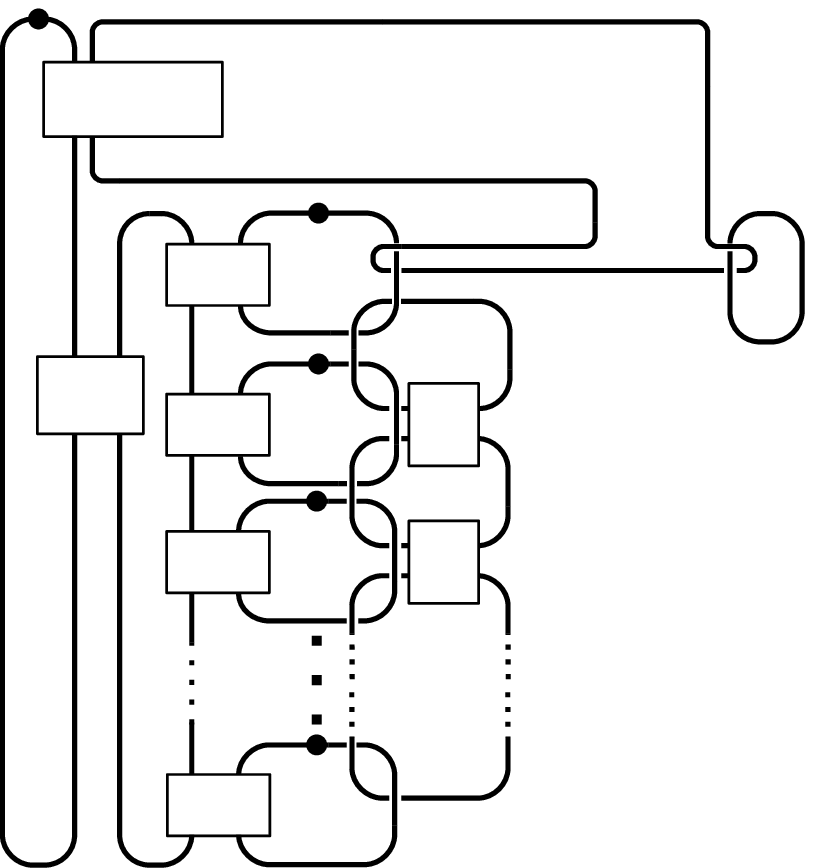}
     \put(30,101){$x_1+\varepsilon$}
     \put(7,87.5){$q\pm1$}

     \put(5.5,52){$\pm1$}

     \put(21,67){$m_1$}
     \put(21,50){$m_2$}
     \put(21,34){$m_3$}
     \put(21,6){$m_s$}

     \put(48.5,49){$b_1$}
     \put(48.5,33){$b_2$}

     \put(87,78){$0$}

     \put(72,26){$\cup$ 3-handle}
     \put(80.2,17){4-handle}
     \end{overpic}
    \caption{}
    \label{Theorem 1.2proof9}
  \end{minipage}
\end{figure}

\begin{figure}[htbp]
  \begin{minipage}[b]{0.3\linewidth}
    \centering
     \begin{overpic}[scale=0.5]
     {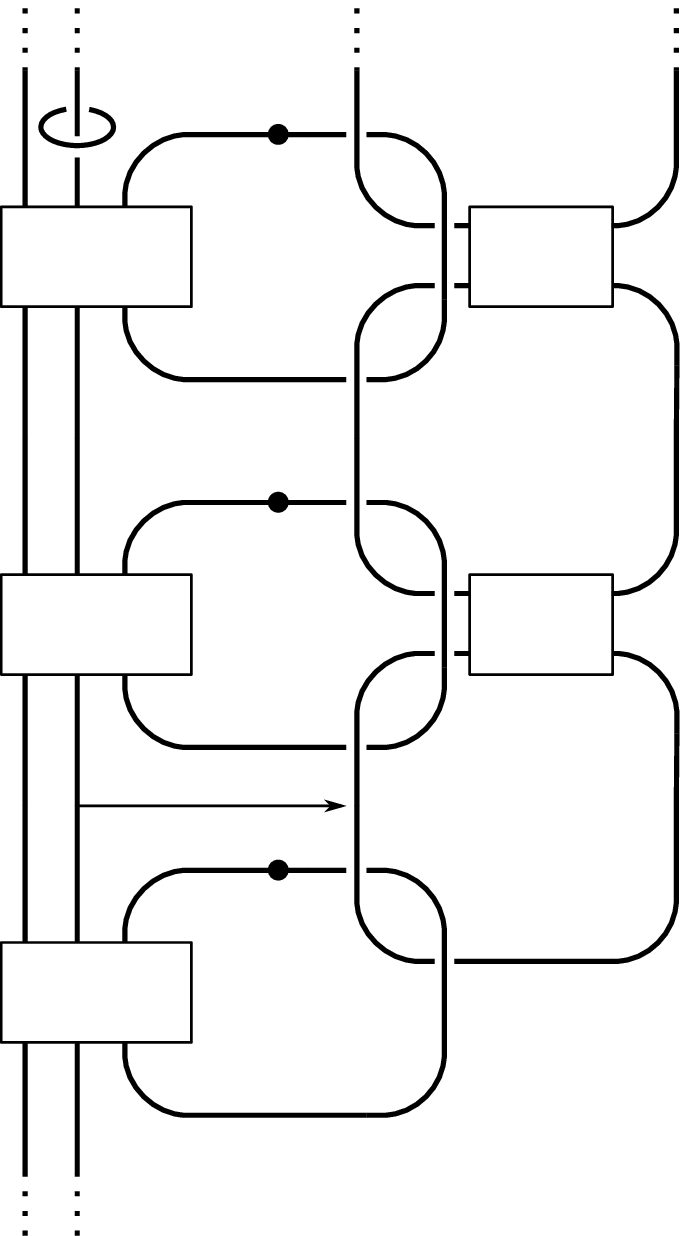}
     \put(2,78.5){$m_{s-2}$}
     \put(2,48.5){$m_{s-1}$}
     \put(2,18.5){$m_s$}

     \put(39,78){$b_{s-3}$}
     \put(39,48){$b_{s-2}$}

     \put(10,90){$0$}
     \end{overpic}
    \caption{}
    \label{Theorem 1.2proof10}
  \end{minipage}
  \begin{minipage}[b]{0.3\linewidth}
    \centering
     \begin{overpic}[scale=0.5]
     {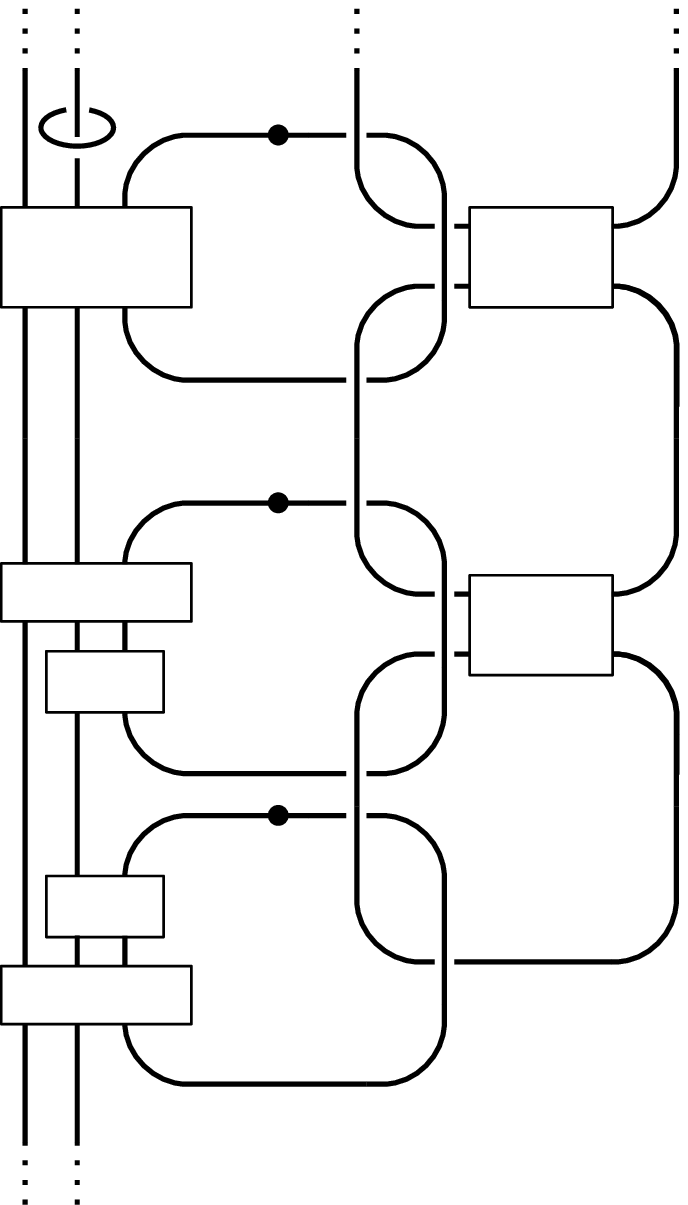}
     \put(2,78){$m_{s-2}$}
     \put(2,50){$m_{s-1}$}
     \put(2,16.7){$m_s$}

     \put(40,78){$b_{s-3}$}
     \put(40,47){$b_{s-2}$}

     \put(10,90){$0$}

     \put(5,42){$-1$}
     \put(7.5,23){$1$}
     \end{overpic}
    \caption{}
    \label{Theorem 1.2proof11}
  \end{minipage}
  \begin{minipage}[b]{0.3\linewidth}
    \centering
     \begin{overpic}[scale=0.5]
     {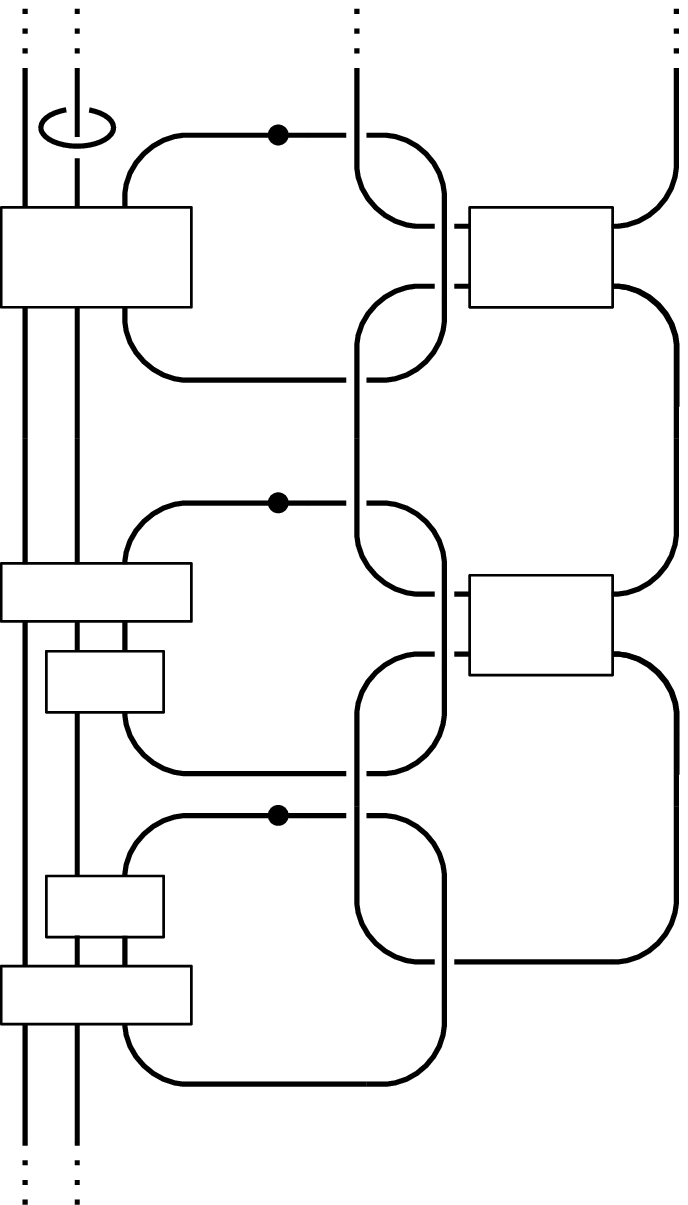}
     \put(2,78){$m_{s-2}$}
     \put(2,50){$m_{s-1}$}
     \put(2,16.7){$m_s$}

     \put(40,78){$b_{s-3}$}
     \put(40,47){$b_{s-2}$}

     \put(10,90){$0$}

     \put(7.5,42){$1$}
     \put(5,23){$-1$}
     \end{overpic}
    \caption{}
    \label{Theorem 1.2proof12}
  \end{minipage}
\end{figure}

\begin{figure}[htbp]
  \begin{minipage}[b]{0.3\linewidth}
    \centering
     \begin{overpic}[scale=0.5]
     {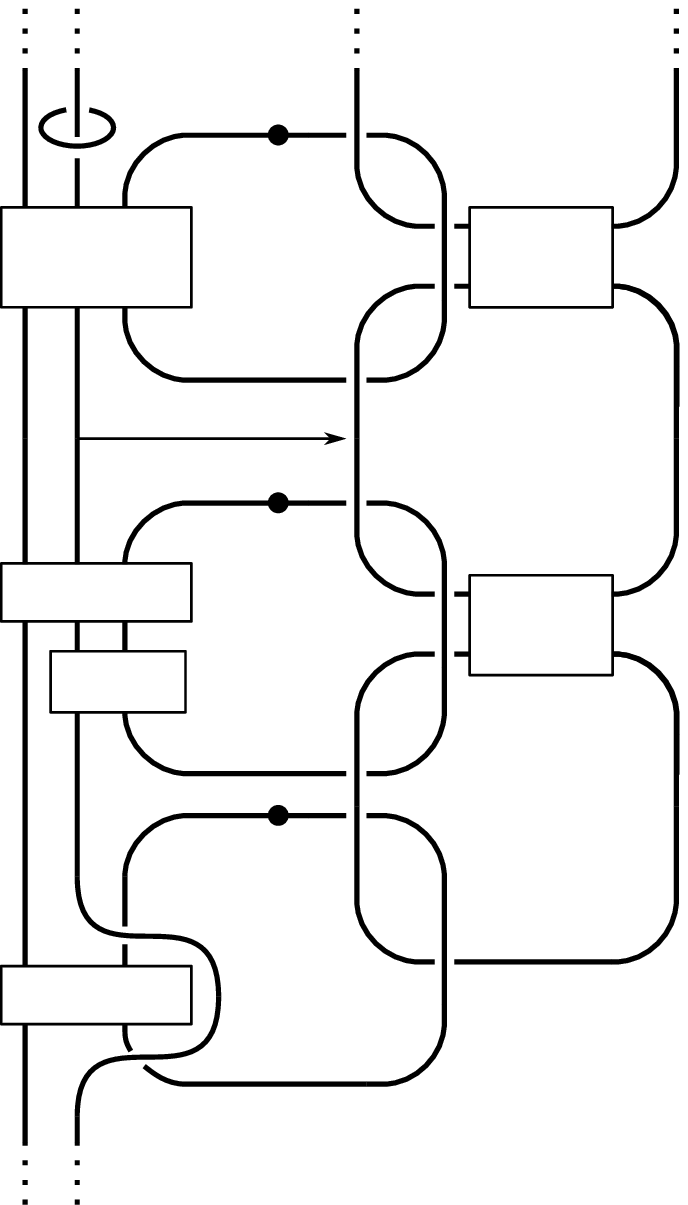}
     \put(2,78){$m_{s-2}$}
     \put(2,50.2){$m_{s-1}$}
     \put(2,16.7){$m_s$}

     \put(40,78){$b_{s-3}$}
     \put(40,47){$b_{s-2}$}

     \put(10,90){$0$}

     \put(6,43){$m_s$}
     \end{overpic}
    \caption{}
    \label{Theorem 1.2proof13}
  \end{minipage}
  \begin{minipage}[b]{0.3\linewidth}
    \centering
     \begin{overpic}[scale=0.5]
     {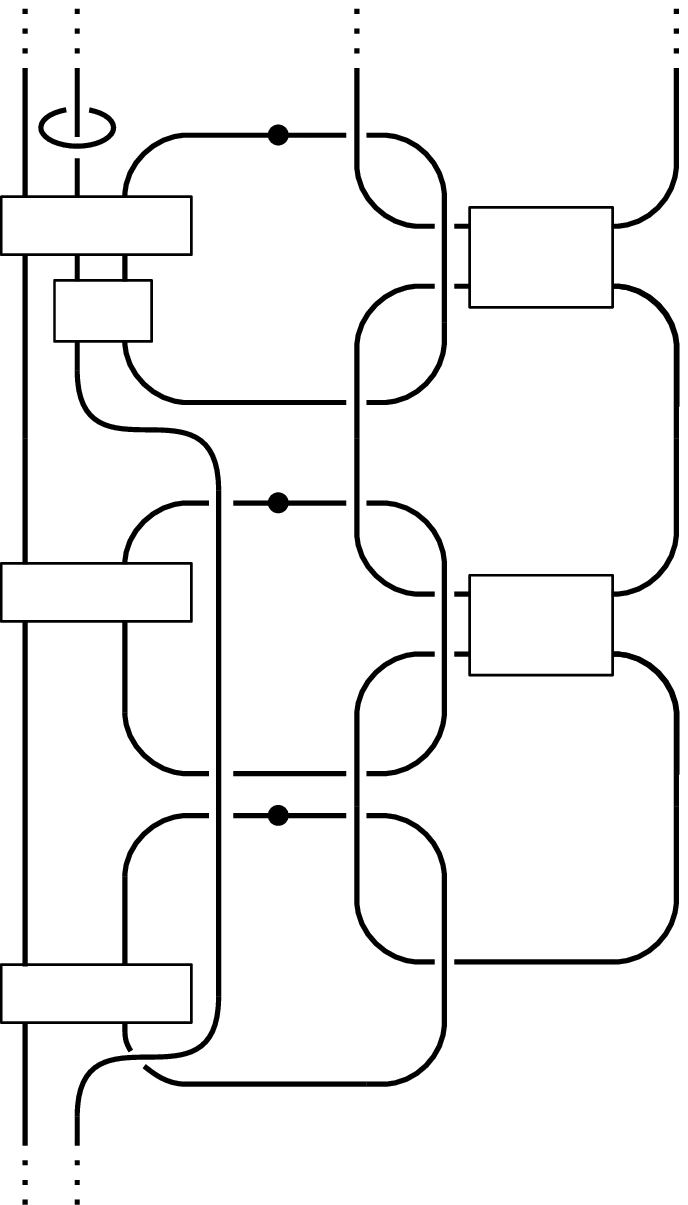}
     \put(2,81){$m_{s-2}$}
     \put(7,73.5){$c$}
     \put(2,50.2){$m_{s-1}$}
     \put(2,16.7){$m_s$}

     \put(40,78){$b_{s-3}$}
     \put(40,47){$b_{s-2}$}

     \put(10,90){$0$}
     \end{overpic}
    \caption{}
    \label{Theorem 1.2proof14}
  \end{minipage}
  \begin{minipage}[b]{0.3\linewidth}
    \centering
     \begin{overpic}[scale=0.5]
     {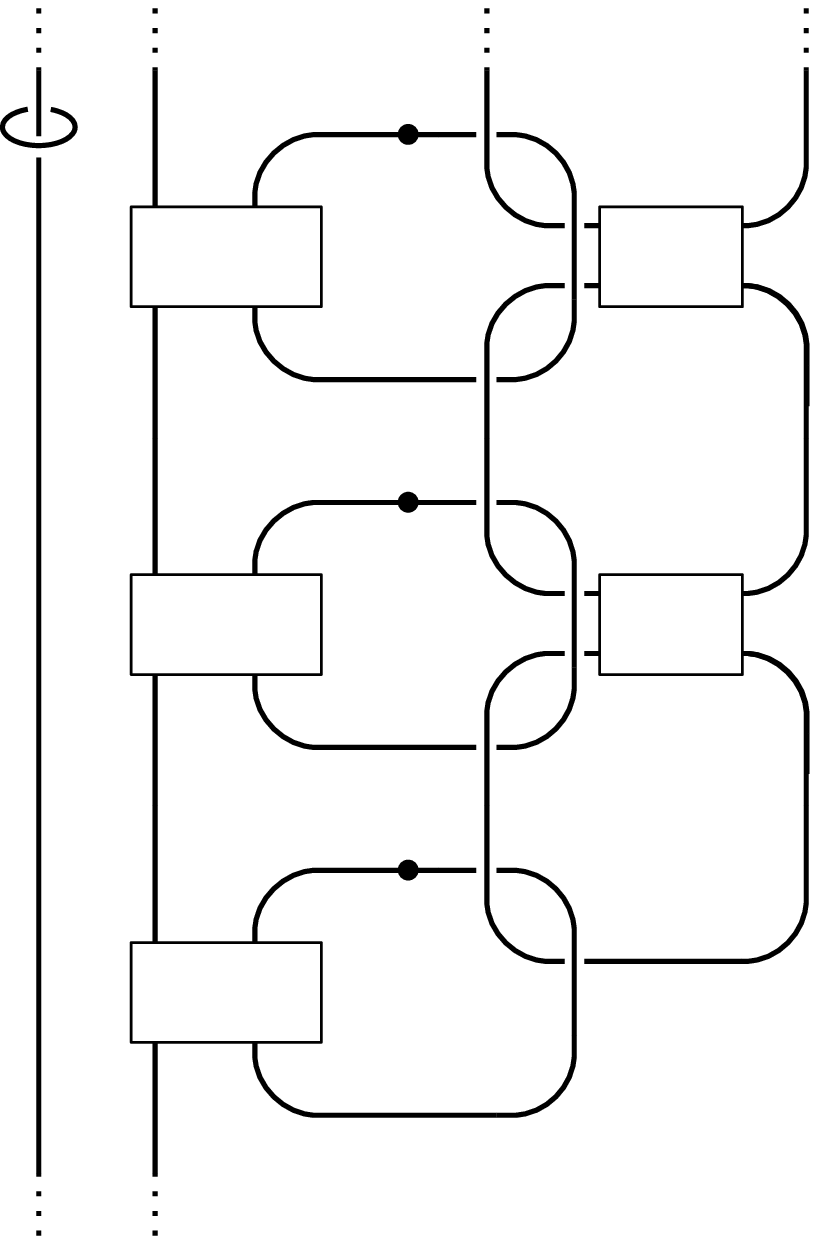}
     \put(12,79){$m_{s-2}$}
     \put(12,49){$m_{s-1}$}
     \put(12,19){$m_s$}

     \put(49.5,78){$b_{s-3}$}
     \put(49.5,48){$b_{s-2}$}

     \put(7,90){$0$}
     \end{overpic}
    \caption{}
    \label{Theorem 1.2proof15}
  \end{minipage}
\end{figure}

\end{document}